\RequirePackage{fix-cm}
\documentclass[smallextended,referee,envcountsect]{svjour3}       
\smartqed  
\usepackage{graphicx}
\usepackage{amsmath, ntheorem}
\usepackage{amssymb}
\usepackage{stackrel}
\usepackage{marvosym, enumerate}
\usepackage{amstext, amscd, latexsym}
\usepackage{amsfonts, ragged2e}
\usepackage{mathrsfs, pgfplots, enumerate, setspace}
\usepackage{subfigure}
\usepackage{cases}
\usepackage{booktabs}
\usepackage{tabularx}
\usepackage[ruled,linesnumbered]{algorithm2e}
\usepackage{float} 
\usepackage[font={bf},textfont=md,labelfont={footnotesize},labelformat={default}]{caption}
\numberwithin{equation}{section}
\smartqed
\usepackage{cite}
\usepackage[colorlinks,linkcolor=blue,anchorcolor=blue,citecolor=blue]{hyperref}
\usepackage{amstext, graphicx, amscd, latexsym, amssymb}
\usepackage{amsfonts, ragged2e}
\usepackage{pgfplots, setspace}
\usepackage{geometry}
\usepackage{latexsym,bm}
\usepackage{float}
\usepackage{threeparttable}
\usepackage{array}
\usepackage{booktabs}
\usepackage{makecell}
\usepackage[justification=centering,labelfont=bf]{caption}
\usepackage{etoolbox}
\usepackage[figuresright]{rotating}
\begin{document}

\title{The Hybridization of Branch and Bound with Metaheuristics for Nonconvex Multiobjective Optimization}

\titlerunning{}        

\author{Wei-tian Wu$^1$ \and Xin-min Yang$^2$ }


\institute{ Wei-tian Wu \at College of Mathematics, Sichuan University, Chengdu 610065, China\\
                    \email{weitianwu@163.com}\\
           \Letter Xin-min Yang \at National Center for Applied Mathematics of Chongqing 401331, China\\
             School of Mathematical Sciences, Chongqing Normal University, Chongqing 401331, China\\
              \email{xmyang@cqnu.edu.cn}}

\date{Received: date / Accepted: date}

\maketitle

\begin{abstract}
A hybrid framework combining the branch and bound method with multiobjective evolutionary algorithms is proposed for nonconvex multiobjective optimization. The hybridization exploits the complementary character of the two optimization strategies. A multiobjective evolutionary algorithm is intended for inducing tight lower and upper bounds during the branch and bound procedure. Tight bounds such as the ones derived in this way can reduce the number of subproblems that have to be solved. The branch and bound method guarantees the global convergence of the framework and improves the search capability of the multiobjective evolutionary algorithm. An implementation of the hybrid framework considering NSGA-II and MOEA/D-DE as multiobjective evolutionary algorithms is presented. Numerical experiments verify the hybrid algorithms benefit from synergy of the branch and bound method and multiobjective evolutionary algorithms.

\end{abstract}

\section{Introduction}
Many optimization problems in the fields of science\textsuperscript{\cite{ref1}} and engineering\textsuperscript{\cite{ref2}} need to consider minimizing multiple and conflicting objectives simultaneously. Researchers refer to such problems as multiobjective optimization problems (MOPs). For them, a set of mathematically equally good solutions which are known as Pareto optimal solutions\textsuperscript{\cite{ref3}} can be identified. Generally, multiobjective optimization methods aim at finding one or a set of Pareto optimal solutions (Pareto set for short) depending on the decision maker's preference.

Following a classification by Hwang and Masud\textsuperscript{\cite{ref4}}, multiobjective optimization methods can be classified into three classes according to the intervention of the decision maker in the solution process: the a priori methods, the interactive methods and the a posteriori methods. For the latter methods, an entirely approximation of the Pareto set is generated first and then the decision maker is supposed to select the most preferred one from candidate solutions. The a posteriori methods are usually considered to be computationally expensive; however, with the vast increase in computer performances, the a posteriori methods have shown some significant advantages. On the one hand, having an approximation of the whole Pareto set means that all the trade-offs associated with the problem are presented to the decision maker, which reinforces the decision maker's confidence on the final decision. On the other hand, no re-optimizations and further interaction are needed for the a posteriori methods.

Most multiobjective evolutionary algorithms\textsuperscript{\cite{ref5,ref6,ref7,ref8,ref9,ref10,ref11,ref12}} can be classified as a posteriori. These algorithms generate an approximate of the whole Pareto front by the way of reproducing operations (crossover, mutation and selection) among the candidate solutions. In fact, the reproducing operations are stochastic search strategies inspired by biological evolution, which makes it difficult to discuss the convergence of MOEAs. Nevertheless, many researchers still choose MOEAs to solve practical problems\textsuperscript{\cite{ref13,ref14,ref15}}, because these algorithms show to be very flexible in handling problems with different types of variables and objectives\textsuperscript{\cite{ref6,ref10,ref12}}.

Other a posteriori methods are based on the divide-and-conquer paradigm, for example, \cite{ref16,ref17,ref18,ref19,ref20}.
Among these approaches, the branch and bound method\textsuperscript{\cite{ref18,ref19,ref20,ref21,ref22,ref23}} have been verified to be effective in handling nonconvex MOPs. The solution process of the branch and bound method implicitly searches for all possible solutions to the problem in a tree structure. The nodes in the tree generate children by subdividing corresponding search region into smaller ones (i.e., branching), with the help of upper and lower bounds for small regions (i.e., bounding), discarding tests are used to prune off the small regions that are provably suboptimal (i.e., pruning). Any branch and bound algorithm for MOPs consists of these three components. For instance, Fern{\'a}ndez and T{\'o}th\textsuperscript{\cite{ref18}} design a branch and bound algorithm for bi-objective optimization problems that employs bisection and uses interval analysis as the bounding method, and proposes three discarding tests to enhance the performance of the algorithm. {\v{Z}}ilinskas et al.\textsuperscript{\cite{ref22}} use trisection, Lipschitz lower bounds, and dominance-based discarding test as part of new branch and bound algorithm. In the branch and bound algorithm presented by Niebling and Eichfelder\textsuperscript{\cite{ref20}}, a new discarding test which combines $\alpha$BB method\textsuperscript{\cite{ref24}} with an extension of Benson's outer approximation techniques\textsuperscript{\cite{ref25}} is presented for their branch and bound algorithm.

The challenge of the branch and bound algorithms raises when facing the exponential-size number of nodes. Hillermeier\textsuperscript{\cite{ref26}} proves that the Pareto set of a continuous MOP is a low-dimensional manifold in the variable space, so that the majority of nodes at the same depth of the search tree are of equal importance, which leads to one branch and bound algorithm increasing the computation time when the discarding test is not able to prevent all nodes from exploration. It is important to note that the performance of the discarding test can be improved by means of obtaining tight bounds. On the other hand, if the tight bounds are not available, the branch and bound algorithm can still be parallelized to reduce computation time, not only because the Pareto solutions may be distributed in any nodes at the same depth level but also because nodes at the same depth level denote disjunct subproblems.

In this paper, we present a parallel algorithmic framework which hybridizes the branch and bound method with an MOEA (PBB-MOEA for short).
The aim of the MOEA is to improve the solution quality including searching for tight upper bounds and inducing tight lower bounds during branch and bound procedure, and further the MOEA is favor of obtaining feasible upper bounds when handling MOPs with nonconvex constraints. Tight bounds derived in this way can reduce the number of subproblems that have to be solved. The branch and bound method not only guarantees the global convergence of the hybrid framework but also improves the search capability of the MOEA, which enables the resulting hybrid algorithm to handle some problems where pure MOEAs may fail. Furthermore, a technology called elitism is integrated in PBB-MOEA in order to reduce computational cost.

This paper is organized as follows. We begin in Section 2 by recalling some basics of multiobjective optimization. Section 3 focuses on new ideas for PBB-MOEA including the use of an MOEA, the improvement of lower bounds and the implementation of the elitism. We devote Section 4 to some theoretical results. In Section 5 PBB-MOEA is compared to MOEAs and a branch and bound algorithm in order to indicates its performance.
\section{Basics of Multiobjective Optimization}
In this paper we consider the following constrained multiobjective optimization problems:
\begin{align*}
 ({\rm MOP}) \qquad\min\limits_{x\in\Omega}\quad F(x)=(f_1(x),\ldots ,f_m(x))^T
\end{align*}
with
\begin{align*}
\Omega=\{x\in\mathbb{R}^n:g_j(x)\geq0,\;j=1,\ldots,p,\;x_i^{(L)}\leq x_i\leq x_i^{(U)},\;i=1,\ldots,n\}
\end{align*}
where $f_i:\mathbb{R}^n\rightarrow \mathbb{R}$ ($i=1,\ldots,m$) are Lipschitz continuous, and $g_j:\mathbb{R}^n\rightarrow \mathbb{R}$ ($j=1,\ldots,p$) are continuous. For a feasible point $x\in\Omega$, the objective vector $F(x)$ is said to be the image of $x$, while $x$ is called the preimage of $F(x)$.

The concept of Pareto optimality for (MOP) is based on the partial ordering of $\mathbb{R}^m$ defined for $z^1, z^2\in \mathbb{R}^m$ by
\begin{align*}
z^1\preccurlyeq z^2&\;\;(z^1\;weakly\;dominates\;z^2)\Longleftrightarrow\;z^1_i\leq z^2_i,\;i\in\{1,\ldots,m\},\\
z^1\preceq z^2&\qquad(z^1 \;dominates\;z^2)\qquad\Longleftrightarrow\;z^1\preccurlyeq z^2\;\wedge\;z^1\neq z^2,\\
z^1\prec z^2&\;\;(z^1\;strictly\;dominates\;z^2)\Longleftrightarrow\;z^1_i< z^2_i,\;i\in\{1,\ldots,m\}.
\end{align*}
When $z^1\npreceq z^2,z^2\npreceq z^1$ and $z^1\neq z^2$, we say that $z^1$ is \emph{indifferent} to $z^2$ ($z^1\sim z^2$). A set $S\subseteq\mathbb{R}^m$ is called \emph{a nondominated set} if for any $z^1,z^2\in S$, $z^1\npreceq z^2$.

A point $x^*\in \Omega$ is said to be \emph{Pareto optimal solution} of (MOP) if there is no point $x\in\Omega$ such that $F(x)\preceq F(x^*)$. The set of all Pareto optimal solutions is called the \emph{Pareto optimal set} and is denoted by $X^*$ or \emph{PS}. The image under the map $F$ of $X^*$ is known as the \emph{Pareto front} of (\rm MOP) and is denoted by \emph{PF}.

The ideal point $z^*\in \mathbb{R}^m$ is defined as the objective vector whose components are the global minimum of each objective function on $\Omega$
\begin{equation*}
z^*={\big (}\min\limits_{x\in\Omega}f_1(x),\min\limits_{x\in\Omega}f_2(x),\ldots,\min\limits_{x\in\Omega}f_m(x){\big )}^T.
\end{equation*}

A nonempty set $B\subseteq\Omega$ is called a \emph{subregion} if $B=\{x\in\mathbb{R}^n:\underline{x}_i<x_i<\overline{x}_i,\;i=1,\ldots,n\}$, and its midpoint is defined as $c(B):=(\overline{x}+\underline{x})/2$. The Euclidean norm is denoted by $\|\cdot\|_2$, $l_1$-norm by $\|\cdot\|_1$ and $l_\infty$-norm by $\|\cdot\|_\infty$. We let $\omega(B)=\overline{x}-\underline{x}$,
then width of $B$ with respect to $l_{(\cdot)}$-norm is defined as $\omega_{(\cdot)}:=\|\omega(B)\|_{(\cdot)}$.

We use $d(a,c)=\|a-c\|_2$ to quantify the distance between two points $a$ and $c$, and the distance between the point $a$ and a nonempty finite set $C$ is defined as $d(a,C):=\min_{c\in C}d(a,c).$ Let $A$ be another nonempty finite set, we define the Hausdorff distance between $A$ and $C$ by
\begin{align*}
d_H(A,C):=\max\{d_h(A,C),d_h(C,A)\},
\end{align*}
where $d_h(A,C)$ is the directed Hausdorff distance from $A$ to $C$ as defined by $d_h(A,C):=\max_{a\in A}d(a,C).$
In addition, the magnitude of a set is denoted by $|\cdot|$.

\section{A Hybrid Framework for MOPs}

In this section, we present the hybrid framework which consists of the branch and bound method and an MOEA. For a simple discussion we introduce the branch and bound algorithm for box-constrained problem first, i.e., we allow $j=0$ in (\rm MOP), this leads to the following problem
\begin{equation*}
 ({\rm P}) \qquad\min\limits_{x\in [a,b]}\quad F(x)=(f_1(x),\ldots ,f_m(x))^T.
\end{equation*}
Algorithm 1 gives an implementation of the basic branch and bound algorithm for (P) (also see \cite{ref18}).

Algorithm 1 can output a covering of the Pareto optimal set and an approximation of the entire Pareto front of (P). At the $k$-th iteration, a subregion $B$ is selected from the subregion collection $\mathcal{B}_k$ and then is bisected into two new subregions $B_1$ and $B_2$ along the $j$-th coordinate, where
$j=\min\big\{\mathop{\arg\max}_{i=1,\ldots,n}\{w_i\}\big\}.$ Algorithm 1 generates a nondominated set $\mathcal{U}^{nds}$ to store upper bounds. For the subregion $B_1$, we check if its upper bound $u(B_1)$ is dominated by any other upper bound in $\mathcal{U}^{nds}$. If so $u(B_1)$ will not be stored into $\mathcal{U}^{nds}$; otherwise, $u(B_1)$ will be add to $\mathcal{U}^{nds}$ and all upper bounds dominated by $u(B_1)$ will be removed. A discarding test is used to check whether a subregion contains any Pareto solution. A common type of discarding test is based on the Pareto dominance relation: \emph{if there exists an upper bound $u\in \mathcal{U}^{nds}$ such that $u$ dominates $l(B_1)$, then $B_1$ does not contain any Pareto solution.} In this case $B_1$ will be stored in $\mathcal{B}_{k+1}$.

In what follows, the set $X^*_B$ denotes the Pareto set of (P) with respect to $\Omega\cap B$. The collection of the Pareto sets of (P) with respect to $\mathcal{B}_k$ is denoted by $X^*_k:=\bigcup_{B\in\mathcal{B}_k}X^*_B$.
The lower and upper bounds for $F(X^*_B)$ will be denoted by $l(B)$ and $u(B)$, respectively, and two bounds are also said to be the bounds with respect to $B$. The lower bound set with respect to $\mathcal{B}_k$ is defined as $\mathcal{L}_k:=\bigcup_{B\in\mathcal{B}_k}l(B)$, while the upper bound set $\mathcal{U}_k$ with respect to $\mathcal{B}_k$ is defined similarly.

\begin{algorithm}
  \SetNoFillComment
  \SetKwInOut{Input}{Input}\SetKwInOut{Output}{Output}
  \Input{(P), termination criterion;}
  \Output{$\mathcal{B}_{k}$, $\mathcal{U}^{nds}$;}
  \BlankLine
  $\mathcal{B}_0\leftarrow [a,b]$, $\mathcal{U}^{nds}\leftarrow \emptyset$, $k=0$\;
  \While{termination criterion is not satisfied}
  {$\mathcal{B}_{k+1}\leftarrow \emptyset$\;
  \While{$\mathcal{B}_k\neq\emptyset$}{
  Select a subregion $B$ in $\mathcal{B}_k$ and remove it from $\mathcal{B}_k$ and
  divide $B$ into $B_1$ and $B_2$\;
  \For{$i=1,2$}
  {Calculate the lower bound $l(B_i)$ and upper bound $u(B_i)$ for $B_i$\;
   \If{$B_i$ can not be discarded}{
        Update $\mathcal{U}^{nds}$ by $u(B_i)$ and store $B_i$ into $\mathcal{B}_{k+1}$\;}}
  }
  $k\leftarrow k+1$.}
  \caption{Basic branch and bound algorithm for (P)}
\end{algorithm}

In what follows, the set $X^*_B$ denotes the local Pareto set 
of (P) with respect to $\Omega\cap B$. The collection of the local Pareto sets of (P) with respect to $\mathcal{B}_k$ is denoted by $X^*_k:=\bigcup_{B\in\mathcal{B}_k}X^*_B$.
The lower and upper bounds for $F(X^*_B)$ will be denoted by $l(B)$ and $u(B)$, respectively, and two bounds are also said to be the bounds with respect to $B$. The lower bound set with respect to $\mathcal{B}_k$ is defined as $\mathcal{L}_k:=\bigcup_{B\in\mathcal{B}_k}l(B)$, while the upper bound set $\mathcal{U}_k$ with respect to $\mathcal{B}_k$ is defined similarly.


\subsection{Improved upper bounds by MOEA}
In single-objective optimization, a global minimal value of objective function over feasible region is restricted from above by the smallest upper bound searched with branch and bound algorithms. Obviously the notion of the smallest upper bound does not make sense to represent an entire Pareto front of a multiobjective optimization problem. Therefore, in the course of the algorithms for multiobjective case, a provisional nondominated set is used to store and update upper bounds found so far. The source of upper bounds may be any points in $\Omega$, more often, the images of the midpoints\textsuperscript{\cite{ref18}} or vertexes\textsuperscript{\cite{ref23}} of subregions can be treated as upper bounds.

In PBB-MOEA, an MOEA is used to calculate upper bounds for each subproblem. And the upper bounds will be stored into the provisional nondominated set $\mathcal{U}^{nds}$. Considering the computational costs of the MOEA, we prefer a mini MOEA to its full version. The ``mini" implies a small initial population and a few generation.

Distinguish from other bounding methods, the advantages of the mini MOEA are as follows: (i) it generates a large quantity of upper bounds to approximate the Pareto front instead of only one for each subproblem, which is beneficial to the discarding test to reduce the number of subregions; (ii) its global search capability is in favor of finding tight upper bounds; (iii) it is able to obtain feasible upper bounds when solving nonconvex constrained MOPs (see Subsection 3.5). The comparative experiments in Section 5 will validate these advantages.

\subsection{Improved lower bounds}
A class of lower bounds based on Lipschitz constant for single-objective optimization problems was proposed by Jones et al. in \cite{ref27}. As they mentioned, the lower bound for the Lipschitz continuous objective function $f:\mathbb{R}^n\rightarrow\mathbb{R}$ with Lipschitz constant $L$ on subregion $B$ can be calculated by
\begin{align}
l(B)=f(c(B))-\frac{L}{2}\omega_{(2)}.\label{E:3.1}
\end{align}

Furthermore, Paulavi{\v{c}}ius and {\v{Z}}ilinskas\textsuperscript{\cite{ref28}} discussed the influence of different norms and corresponding Lipschitz constants on the performance of branch and bound algorithms. As they suggested, a combination of two extreme norms could replace the Euclidean norm in \eqref{E:3.1} in order to enhance the performance, which is as follows
\begin{align}
l(B)=f(c(B))-\frac{1}{2}\min\{L_1\omega_{(\infty)},L_\infty\omega_{(1)}\},\label{E:3.2}
\end{align}
where $L_1=\sup\big\{\|\nabla f(x)\|_1:x\in B\big\}$ and $L_\infty=\sup\big\{\|\nabla f(x)\|_\infty:x\in B\big\}$ are Lipschitz constants. In this paper, Lipschitz constants are estimated by the \emph{natural interval extension}\textsuperscript{\cite{ref27}} of $\nabla f(x)$ on $B$.


As the approach proposed in \cite{ref18,ref20,ref23}, a lower bound for the vector-valued function $F(x)=(f_1(x),\ldots,f_m(x))^T$ on $B$ is the vector whose component is the lower bound for each $f_j$ $(i=1,\ldots,m)$. Therefore, by \eqref{E:3.2}, we choose the point
\begin{align}
l(B)=(l_1,\ldots,l_m)^T\quad with \quad l_i=f_i(c(B))-\frac{1}{2}\min\{L_{i,1}\omega_{(\infty)},L_{i,\infty}\omega_{(1)}\},i=1,\ldots,m\label{E:3.3}
\end{align}
as the lower bound for $F$ on $B$, where $L_{i,1}$ and $L_{i,\infty}$ are Lipschitz constants of $f_i$ corresponding to $l_1$-norm and $l_\infty$-norm, respectively.

For the discarding test, tight lower bounds are favor of reducing the number of subproblems that must be solved. Before discussing the approach to improve the lower bound \eqref{E:3.3}, we introduce a classification of dominating points considered by Skulimowski \textsuperscript{\cite{ref30}}.

\begin{definition}\textsuperscript{\cite{ref30}}\label{de:3.1}
Let $y^*\in\mathbb{R}^m$ be an objective vector.
\begin{enumerate}[\rm (i)]
  \item $y^*$ is called a totally dominating point with respect to $F(X^*_B)$ if $F(X^*_B)\subseteq y^*+\mathbb{R}^m_+$. The set of totally dominating points will be denoted by $\mathcal{TD}(B)$;
  \item $y^*$ is called a partly dominating point with respect to $F(X^*_B)$ if $(y^*+\mathbb{R}^m_+)\cap F(X^*_B)\neq\emptyset$. The set of partly dominating points will be denoted by $\mathcal{PD}(B)$.
\end{enumerate}
\end{definition}

Based on Definition \ref{de:3.1}, now we can introduce two kinds of lower bounds for Pareto front.

\begin{definition}\label{de:3.2}
A nonempty set $\mathcal{L}(B)$ is called an overall lower bound set for $F(X^*_B)$ if for every $x\in X^*_B$ there exists a point $l\in\mathcal{L}(B)$ such that $l \preccurlyeq F(x)$.
\end{definition}

\begin{remark}\label{re:3.3}
Observe that a totally dominating point must be an overall lower bound set, but the converse is not true. This is because an overall lower bound set may consist of multiple partly dominating points. However, an overall lower bound set is a totally dominating point when it is a singleton. Therefore, the lower bound \eqref{E:3.3} is not only a totally dominating point with respect to $F(X^*_B)$ but also an overall lower bound set for $F(X^*_B)$.
\end{remark}
Figure 1(a) illustrates the discuss in Remark \ref{re:3.3} with an example of a bi-objective problem. In what follows, a singleton is still called a totally dominating point with respect to $F(X^*_B)$ and denoted by $l(B)$ if it satisfies Definition \ref{de:3.2}.

By the definition of partly dominating point, we can present a new class of lower bounds, called partial lower bound.
\begin{definition}\label{de:3.4}
Let $z^*$ be the ideal point of $F(X^*_B)$. An objective vector $l\in\mathbb{R}^m$ is called a partial lower bound for $F(X^*_B)$ if $l\in \mathcal{PD}(B)\cap (z^*+\mathbb{R}_+^m\backslash{\{0\}})$. The set of partial lower bounds will be denoted by $\mathcal{P}(B)$.
\end{definition}

The next definition gives an equivalent characterization.
\begin{lemma}\label{le:3.5}
Let $z^*$ be the ideal point of $F(X^*_B)$. An objective vector $l\in\mathbb{R}^m$ is called a partial lower bound for $F(X^*_B)$ if
\begin{enumerate}[\rm (i)]
  \item there exists a point $x\in X^*_B$, such that $l\preccurlyeq F(x)$;
  \item there exists a point $x\in X^*_B$, such that $l\sim F(x)$;
  \item $z^*\preceq l$.
\end{enumerate}
\end{lemma}
\begin{proof}
The proof is obvious and so is omitted. \qed
\end{proof}
\begin{remark}\label{re:3.6}
According to Definitions \ref{de:3.4}, it is easy to see that a partial lower bound is a partly dominating point, thus there exists a subset $\bar{X}\subseteq X^*_B$ such that $l\sim F(\bar{x})$ for every $\bar{x}\in\bar{X}$ and $l\preccurlyeq F(x)$ for every $x\in X^*_B\backslash\bar{X}$. Moreover, we have the following property
\begin{align}
    \mathcal{P}(B)=(F(X^*_B)-\mathbb{R}_+^m)\cap{(z^*+\mathbb{R}_+^m\backslash{\{0\}})}.\label{E:3.4}
\end{align}
\end{remark}
By Remark \ref{re:3.3} and Figure 1(a), we know that an overall lower bound set can consist of multiple partly dominating points, which may be closer to the Pareto front than a totally dominating point. This inspires us whether a totally dominating point can be improved to such an overall lower bound set.

\begin{lemma}\label{le:3.7}
Let a totally dominating point $l(B)=(l_1,\ldots,l_m)^T$ with respect to $F(X^*_B)$ be given. If $\hat{l}(B)=(\hat{l}_{1},\dots,\hat{l}_{m})^T$ is a partial lower bound for $F(X^*_B)$, then the set $$\mathcal{L}(B)=\big\{l^{(i)}:l^{(i)}=(l_1,\ldots,\hat{l}_{i},\ldots,l_m)^T,\;i=1,\ldots,m\big\}$$
is an overall lower bound set for $F(X^*_B)$.
\end{lemma}
\begin{proof}
By Remark \ref{re:3.6}, we know that there exists a set $\bar{X}\subseteq X^*_B$ such that $\hat{l}(B)\sim F(\bar{x})$ for every $\bar{x}\in \bar{X}$ and $\hat{l}(B)\preccurlyeq F(x)$ for every $\bar{x}\in X^*_B\backslash\bar{X}$.

Assume for every $\bar{x}\in \bar{X}$ there exists an index set $\mathcal{I}\subseteq \{1,2,\ldots,m\}$ such that $\hat{l}_{j}<f_j(\bar{x})$ for every $j\in\mathcal{I}$ and $f_k(\bar{x})\leq \hat{l}_{k}$ for every $k\in \{1,2,\ldots,m\}\backslash \mathcal{I}$. Note that $l(B)$ is a totally dominating point with respect to $F(X^*_B)$, it follows that $l_k\leq f_k(\bar{x})\leq \hat{l}_{k}$ for each $k\in \{1,2,\ldots,m\}\backslash \mathcal{I}$. In this way, 
we can construct an objective vector $\bar{l}=(\bar{l}_1,\ldots,\bar{l}_m)^T$ such that $\bar{l}\preceq F(\bar{x})$, where
$$\bar{l}_i=\left\{
          \begin{array}{ll}
              \hat{l}_{i}, & i\in \mathcal{I};\\
              l_i, & i\in\{1,2,\ldots,m\}\backslash \mathcal{I},
          \end{array}\right.
          \label{pan}$$
and further it is easy to see that for all $i\in\mathcal{I}$ such that $l^{(i)}\preccurlyeq \bar{l}\preceq F(\bar{x})$ where $l^{(i)}\in\mathcal{L}(B)$.

On the other hand, we have $\hat{l}(B)\preccurlyeq F(x)$ for each $x\in X^*_B\backslash \bar{X}$. Furthermore, it is easy to see that $l^{(i)}\preceq \hat{l}(B)$ for each $i\in\{1,\ldots,m\}$. Due to the transitivity of Pareto domination relation, we have $l^{(i)}\preceq F(x)$ for each $i\in\{1,\ldots,m\}$ and $x\in X^*_B\backslash \bar{X}$. \qed
\end{proof}

In Lemma \ref{le:3.7}, an overall lower bound set can be obtained by successively exchanging each component of a totally dominating point and a partial lower bound. However, this is not an efficient way to improve a totally dominating point, which can be illustrated with the following example.

\begin{example}
Consider the following multiobjective optimization problem:
$$F(x)=
\begin{pmatrix}
x_1x_2\\
x_1(1-x_2)\\
1-x_1
\end{pmatrix}\quad with\quad x_i\in[0,1],\;i=1,2.$$
It is easy to see that the Pareto front of this problem is contained in the hyperplane $\{(f_1,f_2,f_3)^T:f_1+f_2+f_3=1, f_i\geq 0,i=1,2,3\}$, and its ideal point $z^*$ is $(0,0,0)^T$. Furthermore, according to the Lemma \ref{le:3.5}, we know that the objective vector $l=(0,0.2,0.2)^T$ is a partial lower bound. Based on Lemma \ref{le:3.7}, we can convert $z^*$ into an overall lower bound set
\begin{align*}
\mathcal{L}=\{(0,0,0)^T,(0,0.2,0)^T,(0,0,0.2)^T\}.
\end{align*}
However, it is clear that $\mathcal{L}$ is not an improved lower bound set because $z^*\in\mathcal{L}$.
\end{example}

The above example shows that if the totally dominating point and the partial lower bound have the same value at one coordinate then the totally dominating point will be an element of the overall lower bound set. This means that Lemma \ref{le:3.7} can not improve the totally dominating point in this situation. To avoid the invalid exchange operation, we should check the values of the totally dominating point and the partial lower bound at each coordinate. As a result of the above discussion, we have the following theorem.
\begin{theorem}\label{th:3.9}
Let a totally dominating point $l(B)=(l_1,\ldots,l_m)^T$ and a partial lower bound $\hat{l}(B)=(\hat{l}_1,\ldots,\hat{l}_m)^T$ for $F(X^*_B)$ be given. If $|\tau(\hat{l}(B),l(B))|=0$, then $l(B)$ can be improved to an overall lower bound set
  \begin{align*}
  \mathcal{L}(B)=\{l^{(i)}:l^{(i)}=(l_1,\ldots,\hat{l}_i,\ldots,l_m)^T, \;i=1,\ldots,m\},
  \end{align*}
where $\tau(\hat{l}(B),l(B))=\{i:\hat{l}_i = l_i, i=1,\ldots,m\}$.
\end{theorem}
\begin{proof}
The proof is analogous to the proof of Lemma \ref{le:3.7}.\qed
\end{proof}

Figure 1(b) shows the improvement on a totally dominating point by Theorem \ref{th:3.9} for a bi-objective problem. However, the improvement relies on a known partial lower bound that is not attainable for conflict objectives. In order to find a partial lower bound for Theorem \ref{th:3.9}, we present the following lemma.

\begin{lemma}\label{le:3.10}
Let $z^*\in\mathbb{R}^m$ be the ideal point of $F(X^*_B)$, and let $\hat{z}^*$ be the ideal point of a nonempty compact subset $S$ of $F(B)$. If $\hat{z}^*\in \mathcal{P}(B)$, then we have
\begin{align}
  (\hat{z}^*-\mathbb{R}_+^m\backslash{\{0\}})\cap (F(X^*_B)+\mathbb{R}_+^m)=\emptyset.\label{E:3.5}
\end{align}
\end{lemma}
\begin{proof}
Assume there is an objective vector $z'\in(\hat{z}^*-\mathbb{R}_+^m\backslash{\{0\}})\cap (F(X^*_B)+\mathbb{R}_+^m)$. On the one hand, we have $z'\in\hat{z}^*-\mathbb{R}_+^m\backslash{\{0\}}$, it follows that $z'\preceq \hat{z}^*$. On the other hand, from $z'\in F(X^*_B)+\mathbb{R}_+^m$, we know that there exists a point $x^*\in X^*_B$ such that $F(x^*)\preccurlyeq z'$. Thus we have $F(x^*)\preceq \hat{z}^*$, which is a contradiction to Lemma \ref{le:3.5}. \qed
\end{proof}

\begin{figure}
  \centering
  \subfigure[]{
    \includegraphics[width=0.45\textwidth]{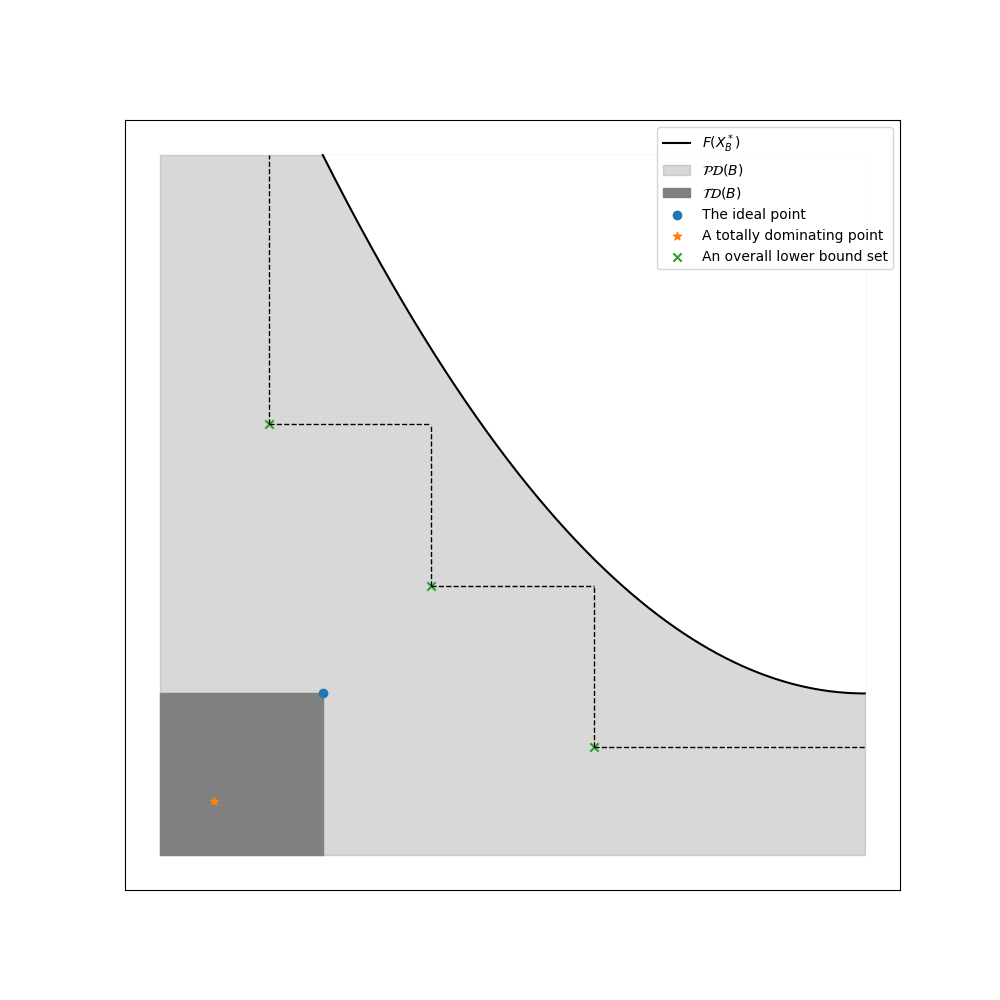}}
  \subfigure[]{
    \includegraphics[width=0.45\textwidth]{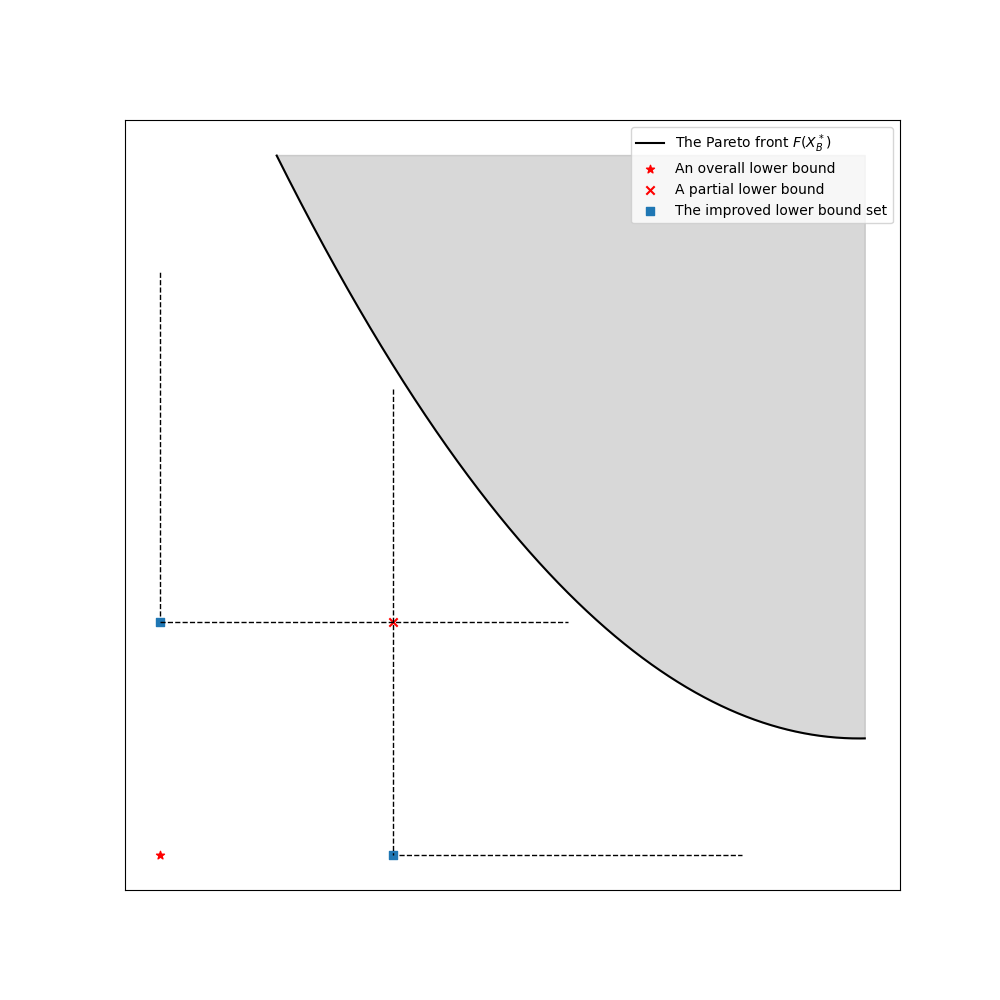}}
  \subfigure[]{
    \includegraphics[width=0.45\textwidth]{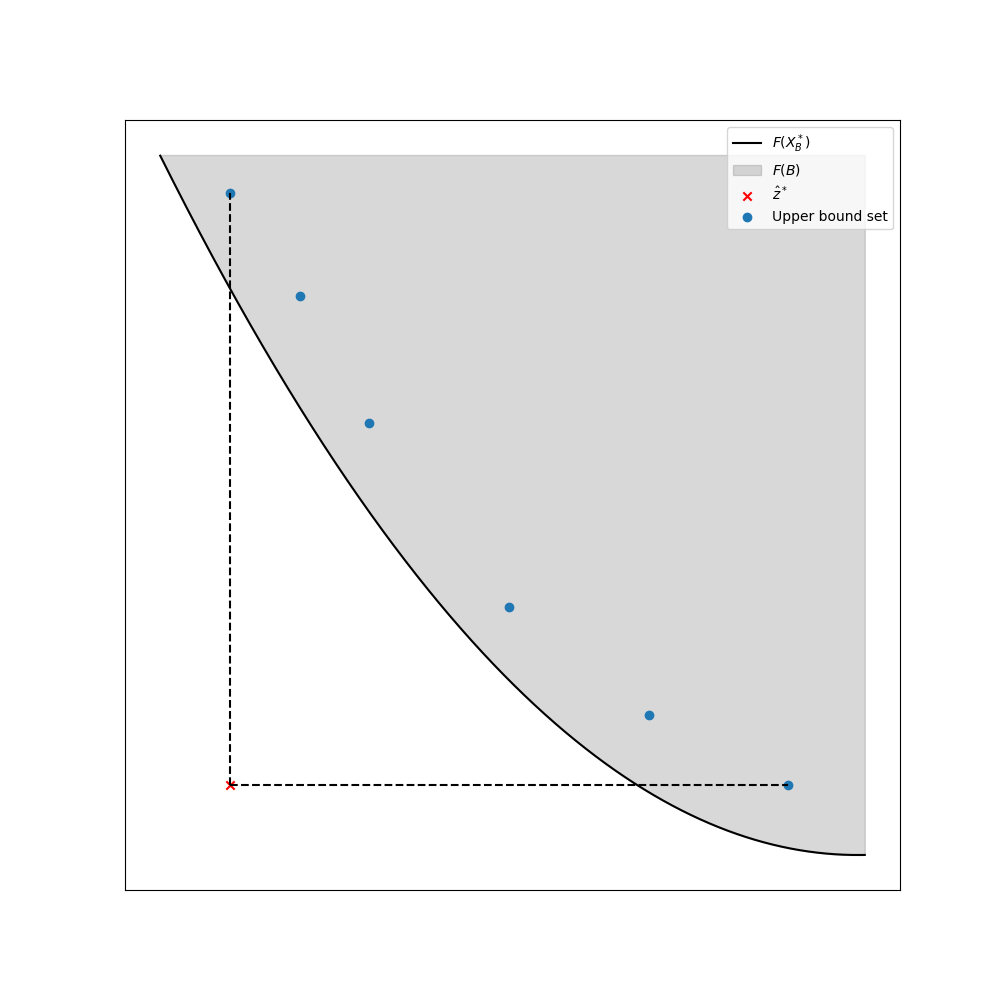}}
  \subfigure[]{
    \includegraphics[width=0.45\textwidth]{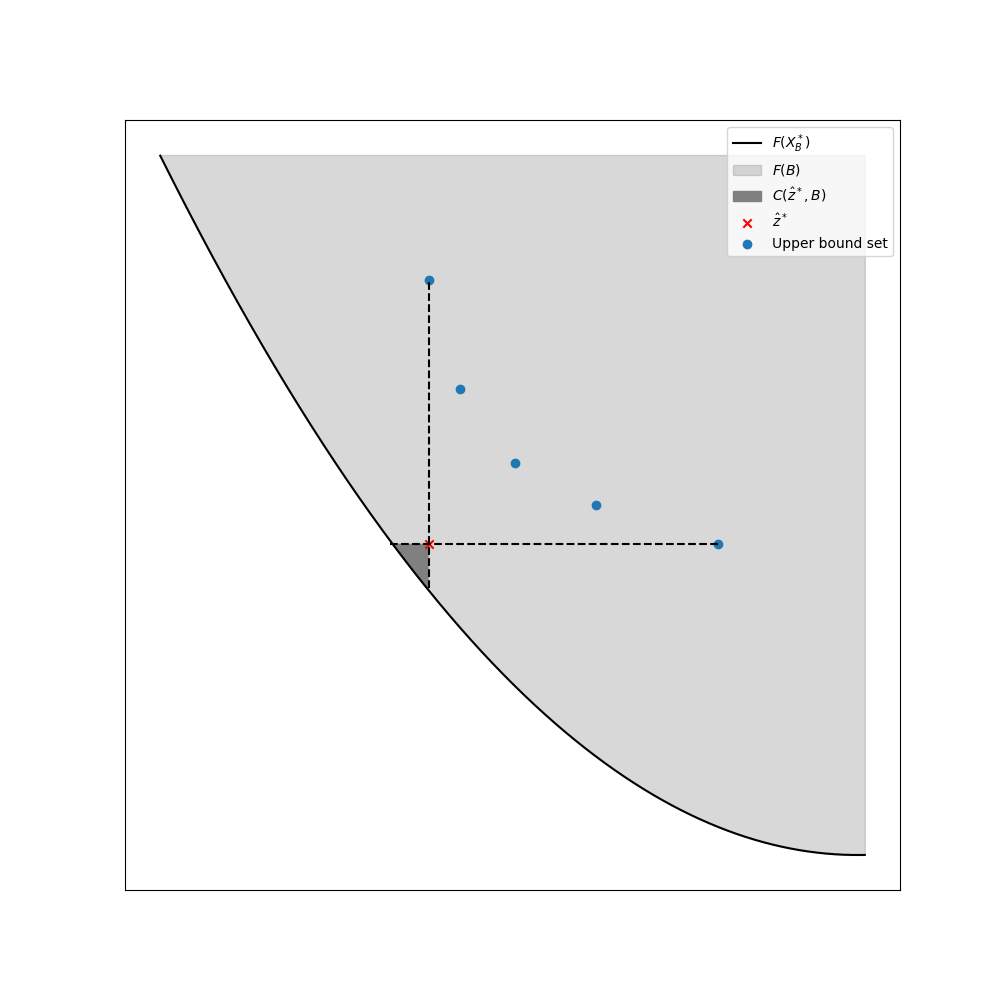}}
  \caption{Details about improving the lower bound: (a) An example of Remark \ref{re:3.3}; (b) Improve an overall lower bound by a partial lower bound; (c) A partial lower bound induced by upper bound set; (d) The ideal point of upper bound set does not satisfy (3.5).}
\end{figure}

Recall that the upper bound set, which has been already found by mini MOEA, is a nonempty compact set. Figure~1(c) shows an example for a partial lower bound induced by upper bound set. Therefore, according to Lemma \ref{le:3.10}, whether the ideal point $\hat{z}^*$ of the upper bound set is a partial lower bound can be determined by checking whether \eqref{E:3.5} is satisfied.
Inspired by \cite{ref31}, this can be done by exploring the search region
$$C(\hat{z}^*,B):= (\hat{z}^*-\mathbb{R}_+^m\backslash{\{0\}})\cap (F(X^*_B)+\mathbb{R}_+^m)$$
which is in order to determine whether $C(\hat{z}^*,B)$ contains feasible points, and if so we assert $\hat{z}^*$ is not a partial lower bound. Such an exploration can be achieved by solving the following mathematical program associated to the search region $C(\hat{z}^*,B)$:
\begin{align*}
P(\hat{z}^*,B)\qquad \min\quad &\sum_{j=1}^m f_j(x)\\
s.t.\quad &F(x)\preceq \hat{z}^*,\\
&x\in B.
\end{align*}
Figure 1(d) shows that $\hat{z}^*$ is not a partial lower bound because $C(\hat{z}^*,B)$ contains feasible points. Note that a minimal solution $\tilde{x}$ of $P(\hat{z}^*,B)$ will be output by a solver even if $C(\hat{z}^*,B)$ is empty, hence we should determine the Pareto dominance relation between $F(\tilde{x})$ and $\hat{z}^*$. If $F(\tilde{x})$ dominates $\hat{z}^*$, then $\hat{z}^*$ is not a partial lower bound; otherwise, $\hat{z}^*$ is a partial lower bound.

Algorithm 2 gives an implementation of Theorem \ref{th:3.9}.

\begin{algorithm}
\caption{Improve lower bounds}
  \SetKwInOut{Input}{Input}\SetKwInOut{Output}{Output}
  \Input{$B$, $l(B)=(l_1,\ldots,l_m)^T$, $\hat{z}^*=(\hat{z}^*_1,\ldots,\hat{z}^*_m)^T$\;}
  \Output{$\mathcal{L}(B)$;}
  $\mathcal{L}(B)\leftarrow \emptyset$\;
  Obtain a minimal solution $\tilde{x}$ of $P(\hat{z}^*,B)$\;
  $\tau(\hat{z}^*,l(B))=\{i:\hat{z}^*_i\neq l(B)_i, i=1,\ldots,m\}$\;

  \If{$F(\tilde{x})\npreceq \hat{z}^*\;and\;|\tau(\hat{z}^*,l(B))|=0$}{
        \lFor{$i=1,\ldots,m$}
            {$\mathcal{L}(B)\leftarrow \mathcal{L}(B)\cup (l_1,\ldots,\hat{z}^*_i,\ldots,l_m)^T$\;}}
\end{algorithm}

\subsection{Elitism}
In PBB-MOEA, the solvers (mini MOEA and the solver for $P(\hat{z}^*,B)$) must be called multiple times in order to obtain tight bounds. As the number of subregions increases, the computational cost is undoubtedly expensive. To circumvent this problem, we introduce a strategy called \emph{elitism} to reduce the number of the solver calls.

For a clearer presentation, we first introduce data tracking in PBB-MOEA. A subregion after subdivision is stored with an index that will be inherited by the data stemming from this subregion. For instance, if we set the index of a subregion to $j$, then the indexes of each upper bound and lower bound with respect to this subregion will be $j$, besides, the ideal point of the upper bound set and the improved lower bound set will also be $j$.

For upper bounds, the elitism is to determine some subregions that are easy to search for tight upper bounds.
The \emph{elite subregion collection} can be backtracked by the indexes stored in $\mathcal{U}^{nds}$ and is denoted by $\mathcal{E}_{U}$. In this way it is only necessary to apply mini MOEA to $\mathcal{E}_{U}$.

For lower bounds, the elitism is to find subregions whose lower bounds are more likely to be dominated by $\mathcal{U}^{nds}$. From the view of geometry, the elite lower bounds are closer to $\mathcal{U}^{nds}$ than others, which can be found in $\mathcal{L}_k$ by Pareto dominance relation and are denoted by $\mathcal{L}^-$. Then, by the indexes stored in $\mathcal{L}^-$, the collection of these ``worst'' subregions $\mathcal{E}_L$ can be constructed. As a result, we only need to apply Algorithm 2 to $\mathcal{E}_L$.

In addition, considering the stability of mini MOEA, PBB-MOEA will reselects $\mathcal{E}_{U}$ from all subregions when a criterion is satisfied. This is called a \emph{repair operation}. More details about the elitism are discussed in subsection 3.5.
\subsection{Discarding test}
PBB-MOEA still uses the dominance relation-based discarding test.

\begin{theorem}\label{th:3.11}
Let $\mathcal{U}^{nds}$ be a nondominated upper bound set with respect to the current subregion collection $\mathcal{B}$, let $\mathcal{L}(B)$ be an overall lower bound set with respect to subregion $B$ generated by Algorithm 2. If for every $l\in\mathcal{L}(B)$ there exists an upper bound $u\in\mathcal{U}^{nds}$ such that $u\preceq l$, then the subregion $B$ does not contain any Pareto solution.
\end{theorem}
\begin{proof}
Because $\mathcal{L}(B)$ is an overall lower bound set with respect to subregion $B$, for each point $x\in B$, there exists a lower bound $l\in\mathcal{L}(B)$ such that $l\preccurlyeq F(x)$. Then it is easy to see that for each point $x\in B$, there exists an upper bound $u\in \mathcal{U}^{nds}$ such that $u\preceq F(x)$. Because $u$ is a feasible objective vector, $B$ does not contain any Pareto solution.
\end{proof}

The discarding test is described in Algorithm 3, where $\mathcal{U}^{nds}$ is a nondominated upper bound set, $\mathcal{B}$ is a subregion collection and $\mathcal{L}$ is a lower bound set. As we plan to track data by indexes, in line 1 the indexes in the subregion collection $\mathcal{B}$ will be collected and stored in the set $\mathcal{I}$. In the for-loop a considered subregion will be removed from $\mathcal{B}$ as well as its lower bound set from $\mathcal{L}$ if each point of its lower bound set is dominated by $\mathcal{U}^{nds}$.

\begin{algorithm}[H]
\caption{\texttt{DiscardingTest}$(\mathcal{L},\mathcal{U}^{nds},\mathcal{B})$}
  \SetKwInOut{Input}{Input}\SetKwInOut{Output}{Output}
  \Input{$\mathcal{L},\mathcal{U}^{nds},\mathcal{B}$\;}
  \Output{$\mathcal{L},\mathcal{B}$\;}
  Collect all indexes from $\mathcal{B}$ and store the indexes in $\mathcal{I}$\;
  \For{each $i\in \mathcal{I}$}{
  \If{for every $l\in\mathcal{L}^{(i)}$ there exists $u\in\mathcal{U}^{nds}$ such that $u\preceq l$}{$\mathcal{L}\leftarrow \mathcal{L}\backslash \mathcal{L}^{(i)}$,
  $\mathcal{B}\leftarrow \mathcal{B}\backslash \mathcal{B}^{(i)}$.}
  }
\end{algorithm}

\subsection{The complete framework}
PBB-MOEA is given in Algorithm 4. The repair operation (lines 7-10) will be performed if the iteration counter $k$ reaches $3n$ and the elitism (lines 12-18) is applied otherwise. PBB-MOEA stops when the Hausdorff distance between $\mathcal{U}^{nds}$ and $\mathcal{L}_k$ drops below the desired accuracy $(0,0.02]$ or iteration counter $k$ reaches $6n$. An approximation of the Pareto optimal set can also be obtained if preimages of $\mathcal{U}^{nds}$ are output by mini MOEA. In addition, at each iteration $\mathcal{U}^{nds}$ will be set to an empty set to release memory for storing new upper bounds.

\begin{algorithm}[htbp]
  \SetNoFillComment
  \SetKwData{flag}{flag}\SetKwData{flags}{flags}\SetKwData{indexes}{Indexes}
  \SetKwFunction{DE}{DE}\SetKwFunction{ILB}{ImproveLowerBound}\SetKwFunction{MiniMOEA}{mini MOEA}
  \SetKwFunction{A}{Algorithm 2}
  \SetKwFunction{DT}{DiscardingTest}
  \SetKwInOut{Input}{Input}\SetKwInOut{Output}{Output}
  \Input{problem (P), termination criteria, repair criterion;}
  \Output{$\mathcal{B}_{k}$, $\mathcal{L}_k$, $\mathcal{U}^{nds}$;}
  \BlankLine
  Set the iteration counter $k$ to $1$, set the \flag of $\mathcal{B}_0$ to 1\; 
  \While{termination criteria are not satisfied}
  {$\mathcal{U}_k\leftarrow\emptyset$, $\mathcal{L}_k\leftarrow\emptyset$, $\mathcal{Z}_k^*\leftarrow \emptyset$, $\mathcal{U}^{nds}\leftarrow \emptyset$\;
  Construct $\mathcal{\hat{B}}_{k}$ by bisecting all subregions in $\mathcal{B}_{k-1}$\; 
  Calculate the lower bound set $\mathcal{L}_k$ for the collection $\mathcal{\hat{B}}_{k}$\;
  \eIf{repair criterion is satisfied}
  {Set all \flags to 1\;
  Obtain $\mathcal{U}_k$ and the ideal point set $\mathcal{Z}_k^*$ by applying \MiniMOEA to $\mathcal{\hat{B}}_{k}$\;
  Obtain the collection of improved lower bound sets $\hat{\mathcal{L}}$ by applying \A to $(\mathcal{\hat{B}}_{k},\mathcal{L}_k,\mathcal{Z}_k^*)$\;
  Update lower bound set $\mathcal{L}_k\leftarrow\hat{\mathcal{L}}$\;}
  {Construct $\mathcal{E}_{U}$ according to \flags\;
  Determine $\mathcal{L}^-$ from $\mathcal{L}_k$\;
  Construct $\mathcal{E}_{L}$ according to \flags in $\mathcal{L}^-$\;
  Obtain the upper bound set $\mathcal{U}_k$ by applying \MiniMOEA to $\mathcal{E}_{U}$\;
  Obtain the ideal point set $\mathcal{Z}^*_k$ by applying \MiniMOEA to $\mathcal{E}_{L}$\;
  Obtain the collection of improved lower bound sets $\hat{\mathcal{L}}$ by applying \A to $(\mathcal{E}_{L},\mathcal{L}^-,\mathcal{Z}^*_k)$\;
  Update lower bound set $\mathcal{L}_k\leftarrow(\mathcal{L}_k\backslash\mathcal{L}^-)\cup\hat{\mathcal{L}}$\;}
  Determine the nondominated upper bound set $\mathcal{U}^{nds}$ from $\mathcal{U}_k$\;
  Update the \flags according to $\mathcal{U}^{nds}$\;
  $\mathcal{L}_k,\mathcal{B}_{k}\leftarrow$\DT{$\mathcal{L}_k,\mathcal{U}^{nds},\mathcal{\hat{B}}_{k}$}\;
  $k\leftarrow k+1$.}
  \caption{PBB-MOEA}
\end{algorithm}

Next we address several issues left open during the implementation of PBB-MOEA.

{\bf Elitism.} Besides an index, each subregion is stored with a boolean called flag stating if it is an elite. The flag is set to 1 if a subregion is an elite and to 0 otherwise.
When a subregion is bisected, two new subregions inherit its flag. In order to perform the repair operation, in line 7, the flags of all subregions are set to 1. In line 21 the flags of the elite subregions whose upper bounds are not in $\mathcal{U}^{nds}$ will be set to 0.

{\bf Parallelization.} To achieve parallelization, in line 4 all subregions are bisected simultaneously. Note that all subregions have the same width, so the bisection will always be along the same coordinate. The parallelization of PBB-MOEA is without communication, i.e., there is no need to design a communication protocol to provide the processors with rules for requesting subproblems and exchanging data. For example, in line 8, algorithm first creates an initial pool of subregions which are corresponding to several subproblems, then distributes each processor one subproblem at a time and requires each processor to perform mini MOEA on the subproblem assigned to it, and finally collects the upper bounds. The parallelization of lines 9, 15, 16 and 17 is similar.

{\bf Constraint handling.} Two difficulties need to be overcome when PBB-MOEA is employed for solving (MOP): (i) determining the feasibility of subregions; (ii) finding feasible upper bounds.

To overcome the first difficulty, the \emph{feasibility test} which is same as the one mentioned in \cite{ref18} will be added between lines 4 and 5.

To overcome the second one, $l_1$ exact penalty function is used in the mini MOEA to calculate the fitness values (see \cite{ref6}). For example, if the Tchebycheff metric is used to calculate the fitness values of individuals for (P):
\begin{align*}
FV(x)=\mathop{\max}\limits_{1\leq i\leq m}\{\lambda_i|f_i(x)-z_i^*|\},
\end{align*}
then for (MOP), the fitness value can be calculated by
\begin{align*}
FV(x)=\mathop{\max}\limits_{1\leq i\leq m}\{\lambda_i|f_i(x)-z_i^*|\}+\rho\mathop{\sum}\limits_{j=1}^{p}|\min\{g_j(x),0\}|,
\end{align*}
where $\rho>0$ is penalty coefficient.

Although the exact penalty function is employed, mini MOEA does not guarantee the feasibility of upper bounds. Hence, in lines 8 and 15 of Algorithm 4, a \emph{filtering operation} should be added to remove upper bounds whose preimages are infeasible. The preimages can be output together with corresponding upper bounds by mini MOEA.

Furthermore, PBB-MOEA still calculate lower bounds by \eqref{E:3.3} for constrained subproblems, because the overall lower bound for unconstrained subproblem is also the overall one for constrained subproblem. In addition, $\mathcal{U}^{nds}$ will not be set to an empty set to increase the quantity of solutions.

\section{Convergence Results}
In this section, we analyze the convergence properties of PBB-MOEA. First, two lemmas which plays a crucial role in the convergence proof are introduced.

\begin{lemma}\label{le:4.1}
Let $B$ be a subregion in the collection $\mathcal{B}_k$. If $\lim\limits_{k\rightarrow\infty}w_k(B)_i=0$ for $i\in\{1,\ldots,m\}$, then for each $x^*\in X^*_B$, we have
\begin{align*}
\lim\limits_{k\rightarrow\infty}d(l(B),F(x^*))=0,
\end{align*}
where $l(B)$ is the lower bound for $F(X^*_B)$ calculated by \eqref{E:3.3}.
\end{lemma}
\begin{proof}
Since $\lim\limits_{k\rightarrow\infty}w_k(B)_i=0$, we have $\lim\limits_{k\rightarrow\infty}\omega_{k,(1)}=0$ and $\lim\limits_{k\rightarrow\infty}\omega_{k,(\infty)}=0$. Thus it follows that
$$\lim\limits_{k\rightarrow\infty}\min\{L_{i,1}\omega_{k,(\infty)},L_{i,\infty}\omega_{k,(1)}\}=0.$$
Then, for each $x^*\in X^*_B$, we have
\begin{align*}
0\leq\lim\limits_{k\rightarrow\infty}d(l(B),F(x^*))\leq\lim\limits_{k\rightarrow\infty}\sqrt{\sum\limits^m_{i=1}(\min\{L_{i,1}\omega_{k,(\infty)},L_{i,\infty}\omega_{k,(1)})^2}=0.
\end{align*}
The second inequality follows from $f_i$ be Lipschitz continuous and \eqref{E:3.3}.
\end{proof}

\begin{lemma}\label{le:4.2}
Let $B$ be a subregion in the collection $\mathcal{B}_k$. If $\lim\limits_{k\rightarrow\infty}w_k(B)_i=0$ for $i\in\{1,\ldots,m\}$, then for each $x^*\in X^*_B$, we have
\begin{align*}
\lim\limits_{k\rightarrow\infty}d(F(x^*),u(B))=0,
\end{align*}
where $u(B)$ is the image of any point in $B$.
\end{lemma}
\begin{proof}
This is a direct consequence of the Lipschitz continuity of $f_i$.
\end{proof}
\begin{remark}\label{re:4.3}
Lemmas \ref{le:4.1} and \ref{le:4.2} show that, for a single subregion, the distance between its Lipschitz lower bound and its Pareto optimal values decreases with the decreasing its width as well as the distance between its upper bounds and the Pareto optimal values.
\end{remark}
\begin{remark}
The conclusion of Lemma \ref{le:4.1} also holds for an improved lower bound set.
\end{remark}

With the help of the preceding lemmas, we can now prove some convergence results.
\begin{theorem}
Let $\{\mathcal{B}_k\}_{k\in\mathbb{N}}$ be the sequence of subregion collection generated by PBB-MOEA, and let $\{\mathcal{L}_k\}_{k\in\mathbb{N}}$ and $\{\mathcal{U}_k\}_{k\in\mathbb{N}}$ be the sequence of lower and upper bound sets, respectively. If $\lim\limits_{k\rightarrow\infty}w_{i}=0$ for $i\in\{1,\ldots,m\}$, then we have
\begin{align}
\lim\limits_{k\rightarrow \infty}d_H(F(X^*_k),\mathcal{L}_{k})=0\label{E:4.1}
\end{align}
and
\begin{align}
\lim\limits_{k\rightarrow \infty}d_H(F(X^*_k),\mathcal{U}_{k})=0.\label{E:4.2}
\end{align}
\end{theorem}

\begin{proof}
First, let us prove \eqref{E:4.1}. It is easy to see that for each $x^*\in X^*_k$, there exists a subregion $B\in\mathcal{B}_k$, such that $x^*\in B$. Then Lemma \ref{le:4.1} is applicable to concluding that
\begin{align*}
0\leq\lim\limits_{k\rightarrow\infty}d(F(x^*),\mathcal{L}_k) \leq \lim\limits_{k\rightarrow\infty}d(F(x^*),l(B))=0,\quad\forall x^*\in X^*_k,
\end{align*}
thus we obtain $\lim\limits_{k\rightarrow\infty}d_h(F(X^*),\mathcal{L}_{k})=0$.

On the other hand, for each $l\in \mathcal{L}_k$, there exists a subregion $B\in\mathcal{B}_k$, such that $l(B)=l$. Using Lemma \ref{le:4.1}, we have
\begin{align*}
0\leq\lim\limits_{k\rightarrow\infty}d(l,F(X^*_k))\leq\lim\limits_{k\rightarrow\infty}d(l(B),F(\hat{x}))=0,\quad \forall \hat{x}\in X^*_B.
\end{align*}
which means that $\lim\limits_{k\rightarrow\infty}d_h(\mathcal{L}_{k},F(X^*_k))=0$. As a result, we prove that $$\lim\limits_{k\rightarrow\infty}d_H(\mathcal{L}_{k},F(X^*_k))=0.$$
The proof of \eqref{E:4.2} is similar to the proof above and so is omitted.
\end{proof}

\begin{corollary}
Let $\{\mathcal{B}_k\}_{k\in\mathbb{N}}$ be the sequence of subregion collection generated by PBB-MOEA, and let $\{\mathcal{L}_k\}_{k\in\mathbb{N}}$ and $\{\mathcal{U}_k\}_{k\in\mathbb{N}}$ be the sequence of lower and upper bound sets, respectively. If $\lim\limits_{k\rightarrow\infty}w_{i}=0$ for $i\in\{1,\ldots,m\}$, then we have
$$\lim\limits_{k\rightarrow \infty}d_H(\mathcal{L}_{k},\mathcal{U}_{k})=0.$$
\end{corollary}
\begin{proof}
The conclusion can be derived from Lemmas \ref{le:4.1} and \ref{le:4.2}.
\end{proof}

Now, we show that PBB-MOEA creates a covering of the Pareto set in each iteration.

\begin{lemma}\label{le:4.7}
Let $\{\mathcal{B}_k\}_{k\in\mathbb{N}}$ be a sequence of subregion collections generated by PBB-MOEA. Then, for the Pareto set $X^*$ of (P), we have
\begin{align*}
X^*\subset\cdots\subset \bigcup_{B\in\mathcal{B}_k} B\subset\cdots\subset\bigcup_{B\in\mathcal{B}_1} B \subset \bigcup_{B\in\mathcal{B}_0} B.
\end{align*}
\end{lemma}
\begin{proof}
This conclusion is guaranteed by Theorem \ref{th:3.11} and the way $\mathcal{B}_k$ is constructed.
\end{proof}

Based on Lemma \ref{le:4.7}, we have the following theorem.
\begin{theorem}\label{th:4.8}
Let $\{\mathcal{B}_k\}_{k\in\mathbb{N}}$ be a sequence of subregion collections generated by PBB-MOEA. For each Pareto solution $x^*\in X^*$, for all iteration counters $k\in\mathbb{N}$, there exists a subregion $B(x^*,k)\in\mathcal{B}_k$, such that
$x^*\in B(x^*,k)$. Furthermore, $X^* \subset \bigcap\limits_{k=0}^{\infty}\bigcup\limits_{x\in X^*}B(x,k).$
\end{theorem}
\begin{proof}
The existence of $B(x^*,k)$ is generated by Lemma \ref{le:4.7}. The first conclusion can be proved by mathematical induction. The second conclusion is derived from the first one.
\end{proof}

\begin{theorem}\label{th:4.9}
Let $\{\mathcal{B}_k\}_{k\in\mathbb{N}}$ be the sequence of subregion collections generated by PBB-MOEA. If $\lim\limits_{k\rightarrow\infty}w_k(B)_i=0$ for $i\in\{1,\ldots,m\}$ and $B\in\mathcal{B}_k$, then there must exist a subcollection $\mathcal{B}^*_k\subset\mathcal{B}_k$ for all $k\in\mathbb{N}$, such that the following hold:
\begin{align}
&\lim\limits_{k\rightarrow \infty}d_H(X^*,\mathcal{B}^*_k)=0,\label{E:4.3}\\
&\lim\limits_{k\rightarrow \infty}d_H(F(X^*),\mathcal{U}(\mathcal{B}^*_k))=0,\label{E:4.4}\\
&\lim\limits_{k\rightarrow \infty}d_H(F(X^*),\mathcal{L}(\mathcal{B}^*_k))=0,\label{E:4.5}
\end{align}
where $\mathcal{U}(\mathcal{B}^*_k):=\bigcup_{B\in\mathcal{B}^*_k}u(B)$ and $\mathcal{L}(\mathcal{B}^*_k):=\bigcup_{B\in\mathcal{B}^*_k}l(B)$.

\end{theorem}
\begin{proof}
Setting $\mathcal{B}^*_k=\bigcup_{x\in X^*}B(x,k)$. By Theorem \ref{th:4.8}, it is easy to see that 
\begin{align*}
\lim\limits_{k\rightarrow \infty}d(x^*,\mathcal{B}^*_k) =0 \quad \forall x^*\in X^*,
\end{align*}
this entails that $\lim\limits_{k\rightarrow \infty}d_h(X^*,\mathcal{B}^*_k)=0$. On the other hand, for each $x\in\mathcal{B}_k^*$, there exists a subregion $B(\hat{x}^*,k)\in\mathcal{B}_k$, such that $x\in B(\hat{x}^*,k)$. Therefore, we have
\begin{align*}
0\leq\lim\limits_{k\rightarrow \infty}d(x,X^*)\leq\lim\limits_{k\rightarrow \infty}d(x,\hat{x}^*)=\lim\limits_{k\rightarrow \infty}\sqrt{\sum_{i=1}^m w_k(B(\hat{x}^*,k))_i^2} =0,
\end{align*}
which means that $\lim\limits_{k\rightarrow \infty}d_h(\mathcal{B}_k^*, X^*)=0$. This proves \eqref{E:4.3}.

Now, we turn to the relation between the lower bound set of $\mathcal{B}^*_k$ and Pareto front. By the construction of $\mathcal{B}^*_k$ and Lemma \ref{le:4.2}, we have the following inequality
\begin{align*}
0\leq d(F(x^*),\mathcal{U}(\mathcal{B}^*_k))\leq d(F(x^*),u(B(x^*,k)))=0\quad \forall x^*\in X^*.
\end{align*}
On the other hand, for each $u\in \mathcal{U}(\mathcal{B}^*_k)$, there exists a point $\hat{x}\in X^*$, such that $u(B(\hat{x},k))=u$. Using Lemma \ref{le:4.2} again, we obtain
\begin{align*}
0\leq d(u,F(X^*))\leq d(u(B(\hat{x},k)),F(\hat{x}))=0\quad \forall u\in \mathcal{U}(\mathcal{B}^*_k).
\end{align*}
This shows \eqref{E:4.4}. The proof of \eqref{E:4.5} is similar to the above proof.
\end{proof}

Theorem \ref{th:4.9} shows that, at each iteration, there must exist a subcollection of the subregion collection, such that a sequence constructed by the subcollection per iteration converges to the Pareto set. In practice, it is hard to determine such a subcollection for PBB-MOEA, thus we try to find the elite subregions to simulate the subcollection. Intuitively, the collection of elite subregions will be very similar to the subcollection if the width of subregions is small enough.

\section{Experimental Results}

PBB-MOEA has been implemented in Python 3.8 with fundamental packages like numpy, scipy and multiprocessing. All experiments have been done on a computer with Intel(R) Core(TM) i7-10700 CPU and 32 Gbytes RAM on operation system WINDOWS 10 PROFESSIONAL.

We hybridize the branch and bound method with MOEA/D-DE\textsuperscript{\cite{ref7}} and NSGA-II\textsuperscript{\cite{ref10}}, respectively. The resulting algorithms are called PBB-MOEA/D and PBB-NSGAII, respectively. For all test instances, the population size of two mini MOEAs is set to 10 and the maximum number of generations is set to 20. Other parameter settings of mini MOEAs can be found in \cite{ref7,ref10}.

{\bf Test instance 5.1}
Consider the following optimization problem from \cite{ref23}:
$$F(x)=
\begin{pmatrix}
x_1\\
\min(|x_1-1|,1.5-x_1)+x_2+1
\end{pmatrix}\quad with\quad x_i\in[0,2],\;i=1,2.$$
The Pareto optimal set of this problem is disconnected, which leads to its Pareto front consisting of two segments: one joins the points (0, 2) and (1, 1), and the other, the points (1.5, 1) and (2, 0.5).

Figure 1 shows the results of PBB-MOEA/D on this instance. After 12 iterations, PBB-MOEA/D outputs 96 subregions covering the Pareto set. These subregions are visualized as grey boxes in Figure 2(a). In Figure 2(b), a total of 479 upper bounds (blue dots) found by the mini MOEA and 96 lower bounds (red dots) enclose the Pareto front from above and below.
\begin{figure}[H]
  \centering
  \subfigure[]{
    \includegraphics[width=0.45\textwidth]{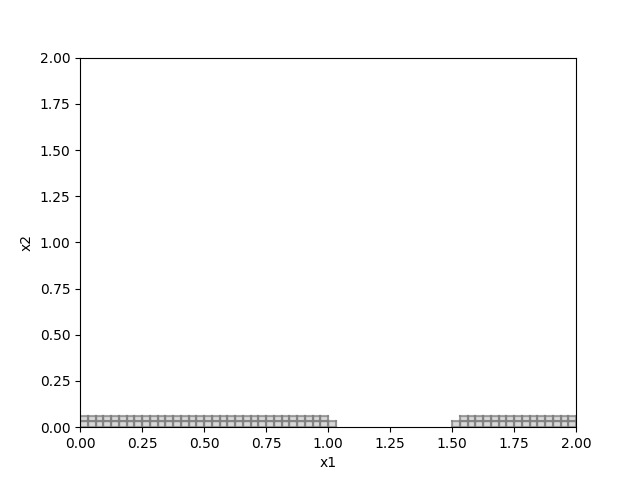}}
  \subfigure[]{
    \includegraphics[width=0.5\textwidth]{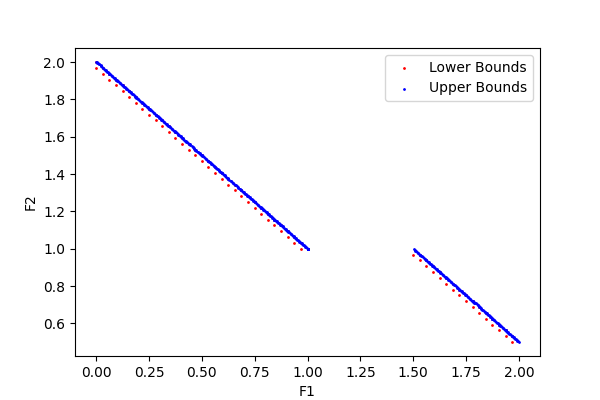}}
  \caption{The results of test instance 5.1.: (a) A covering of the PS of test instance 5.1 generated by PBB-MOEA/D; (b) Lower and upper bounds for the PF of test instance 5.1 generated by PBB-MOEA/D.}
\end{figure}

In the following, we use two instances to illustrate the difference between PBB-MOEA/D and a basic branch and bound (BB) algorithm. The basic BB uses the midpoints of subregions to calculate upper bounds and uses \eqref{E:3.3} to calculate lower bounds. Here the number of subregions (BNV) is a performance measure of tightness of bounds.

{\bf Test instance 5.2} Consider the biobjective optimization problem from \cite{ref32}:
$$F(x)=
\begin{pmatrix}
1-e^{-\sum\limits_{i=1}^{n}(x_i-\frac{1}{\sqrt{n}})^2}\\
1-e^{-\sum\limits_{i=1}^{n}(x_i+\frac{1}{\sqrt{n}})^2}
\end{pmatrix}\quad with\quad x_i\in[-2,2],\;n=3.$$

Figure 3 gives a visual comparison of the two algorithms on test instance 5.2. Figure 3(a) shows the evolution of BNVs versus the number of iterations, although the basic BB algorithm can handle this instance well, the curve of PBB-MOEA/D mostly lie to the below of the basic BB algorithm, indicating that the performance of PBB-MOEA/D is still slightly better. Figure 3(b) presents the upper bounds found by two algorithms in the 9-th iteration, respectively. It can be easily seen that the upper bounds found by PBB-MOEA/D (blue dots) mostly dominates the ones found by basic BB algorithm (red cubes), which implies that blue dots are tighter. Figures 3(c) and 3(d) illustrate the influence of different upper and lower bounds on the discarding test. PBB-MOEA/D finds a total of 884 upper bounds, which is obviously more than upper bounds in Figure 3(d). As a result, PBB-MOEA/D is able to remove more subregions in the same discarding test.

\begin{figure}[H]
  \centering
  \subfigure[]{
    \includegraphics[width=0.45\textwidth]{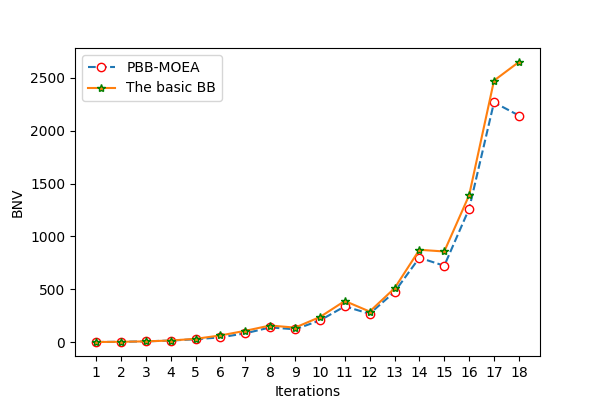}}
  \subfigure[]{
    \includegraphics[width=0.43\textwidth]{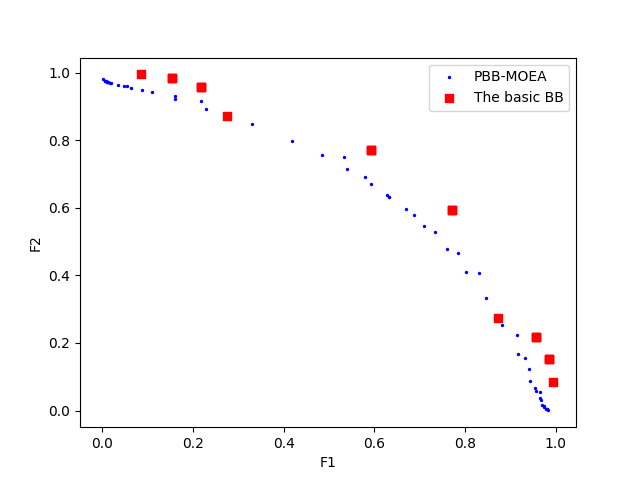}}
  \subfigure[]{
    \includegraphics[width=0.45\textwidth]{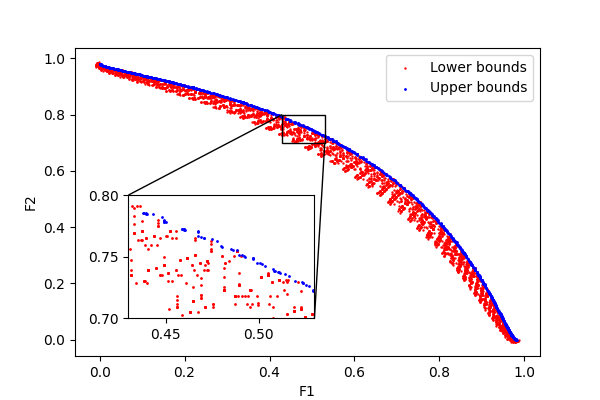}}
  \subfigure[]{
    \includegraphics[width=0.45\textwidth]{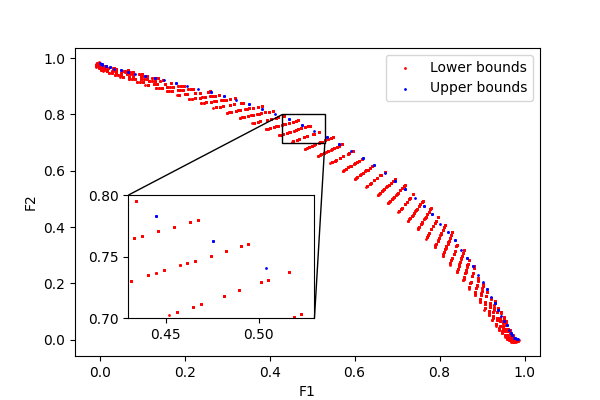}}
  \caption{PBB-MOEA/D versus the basic BB algorithm on test instance 5.2.:(a) The curves of BNVs of PBB-MOEA and the basic BB algorithm; (b) Upper bounds generated by two algorithms at the 9-th iteration; (c) Bounds generated by PBB-MOEA/D; (d) Bounds generated by the basic BB algorithm.}
\end{figure}

{\bf Test instance 5.3} Consider the biobjective optimization problem from \cite{ref10} and the dimension of the variable space $n\in\mathbb{N}$ can be arbitrarily chosen:
$$F(x)=
\begin{pmatrix}
 x_1\\
g(x)(1-(\frac{f_1(x)}{g(x)})^2)
\end{pmatrix}$$
where $g(x)=1+\frac{9(\sum\limits_{i=2}^n x_i)}{n-1}$ and $x_i\in[0,1],\;n\in\mathbb{N}.$ This instance is well-known ZDT2 which has a nonconvex Pareto front, and its Pareto set is $x_1\in[0,1]$, $x_i=0$, $i\geq2$.

The illustrations in Figure 4 show the results for this instance with $n=10$. Figures 4(a) and 4(b) show that PBB-MOEA/D can handle this instance very well. In Figure 4(c), the BNVs grow exponentially when using the basic BB algorithm, which means that adequate bounds are not found, resulting in a failure of solving this instance. In contrast, with the help of the MOEA, PBB-MOEA/D can reduce the number of subregions successfully, which indicates its superior performance.

\begin{figure}[H]
  \centering
  \subfigure[]{
    \includegraphics[width=0.30\textwidth]{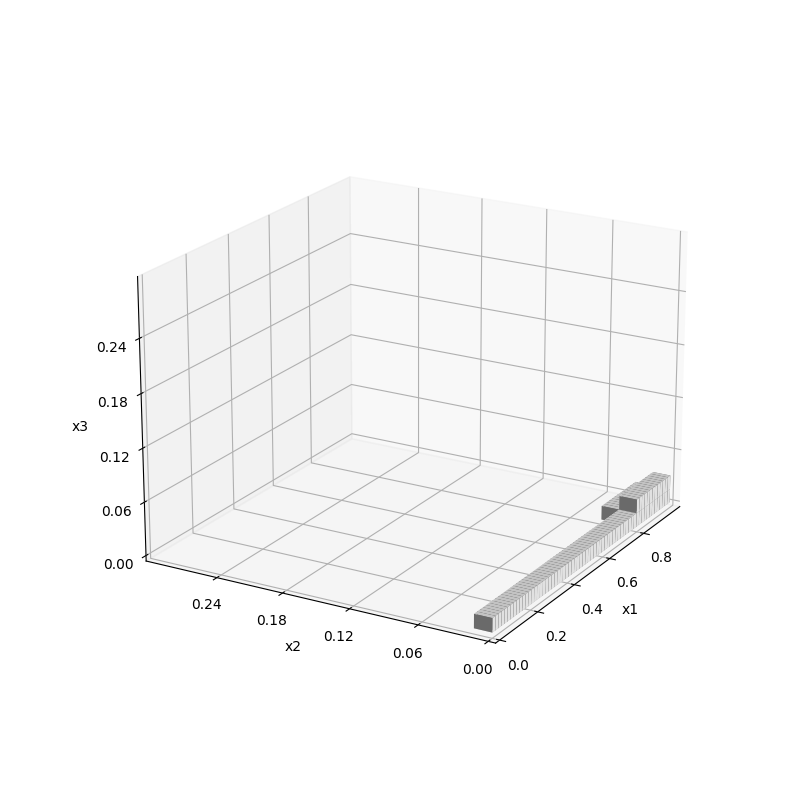}}
  \subfigure[]{
    \includegraphics[width=0.32\textwidth]{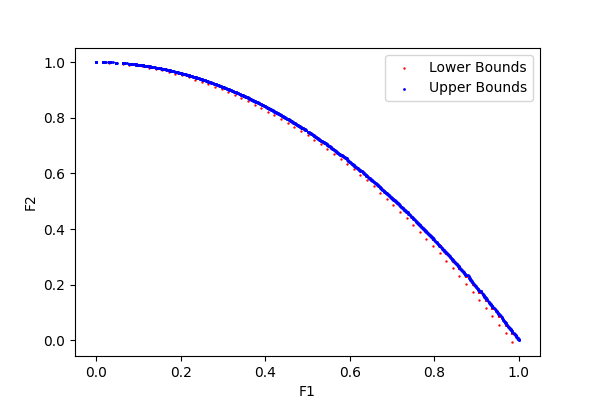}}
  \subfigure[]{
    \includegraphics[width=0.33\textwidth]{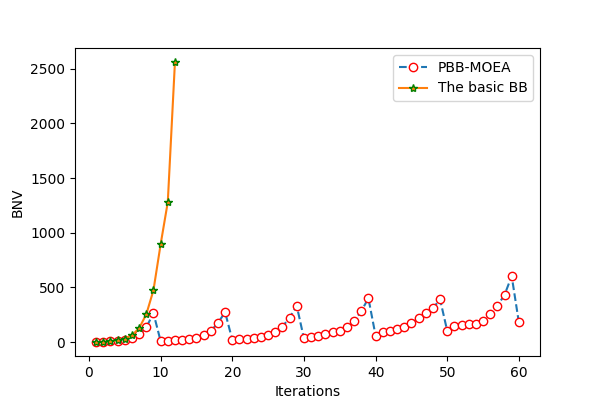}}
  \caption{The results of test instance 5.3 with $n=10$: (a) A covering of the PS of test instance 5.3 generated by PBB-MOEA/D; (b) Lower and upper bounds for the PF of test instance 5.3 generated by PBB-MOEA/D; (c) The curves of BNVs of PBB-MOEA and the basic BB algorithm.}
\end{figure}

One may ask that MOEAs can also handle the above instances well, is it necessary to consider the hybridization of the branch and bound method and MOEAs? In the following, we compare PBB-MOEA with two well-known MOEAs: NSGA-II and MOEA/D-DE.

{\bf Test instance 5.4} Consider the biobjective optimization problem from \cite{ref33}:
$$F(x)=
\begin{pmatrix}
\sum\limits_{i=1}^n x_i\\
1-\sum\limits_{i=1}^n (1-w_i(x_i))
\end{pmatrix}$$
where $w_j(z)=
          \left\{
          \begin{array}{ll}
              0.01e^{-(z/20)^{2.5}} & \mbox{if }j=1,2\\
              0.01e^{-(z/15)} & \mbox{if }j>3
          \end{array}\right.
          \label{pan}$ and $x_i\in[0,40]$, $n\in\mathbb{N}$. \\
          This test instance is a multimodal problem, i.e., different Pareto solutions may have the same function value.

Figure 5 provides the comparison of three algorithms on test instance 5.4 with $n=3$. For both NSGA-II and MOEA/D-DE, the quality of a solution set is evaluated in the objective space, the distribution of solutions in the variable space does not received enough attention. As a result, though they are able to approximate the whole Pareto front well in Figures 5(d) and 5(e), only a part of the Pareto set is approximated by solutions they find which are shown in Figures 5(a) and 5(b). In contrast, the branch and bound method focus on solving the problem from the variable space, thus Figures 5(c) and 5(f) show that PBB-MOEA/D can approximate not only the Pareto front but also the whole Pareto set.
\begin{figure}
  \centering
  \subfigure[]{
    \includegraphics[width=0.32\textwidth]{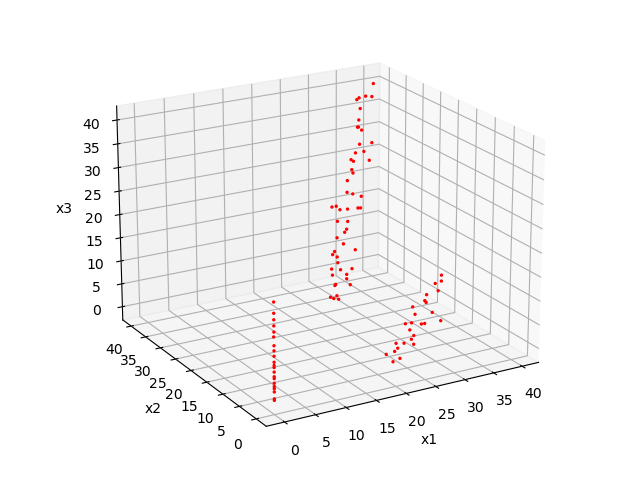}}
  \subfigure[]{
    \includegraphics[width=0.32\textwidth]{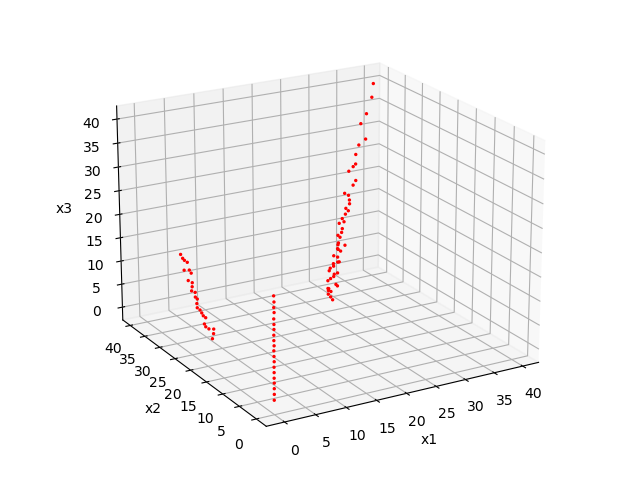}}
  \subfigure[]{
    \includegraphics[width=0.32\textwidth]{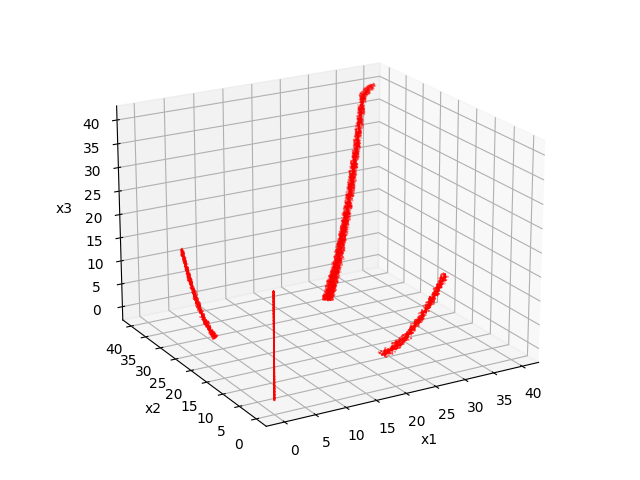}}

  \subfigure[PF of NSGA-II]{
    \includegraphics[width=0.32\textwidth]{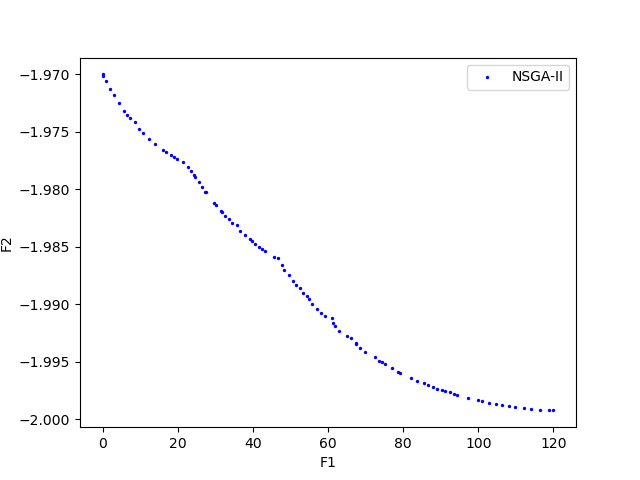}}
  \subfigure[]{
    \includegraphics[width=0.32\textwidth]{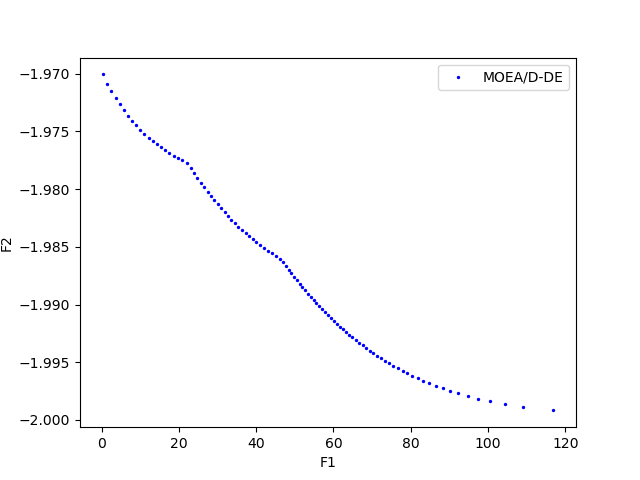}}
  \subfigure[]{
    \includegraphics[width=0.32\textwidth]{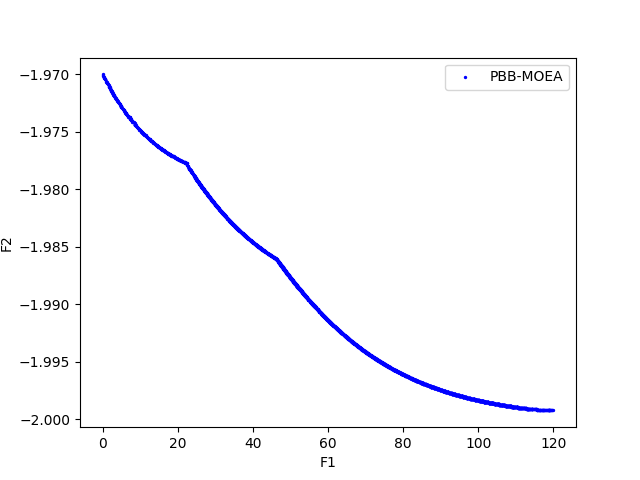}}
  \caption{PBB-MOEA/D versus two MOEAs on test instance 5.4 with $n=3$: (a) The solutions generated by NSGA-II; (b) The solutions generated by MOEA/D-DE; (c) The solutions generated by PBB-MOEA/D; (d) The upper bounds generated by NSGA-II; (e) The upper bounds generated by MOEA/D-DE; (f) The upper bounds generated by PBB-MOEA/D.}
\end{figure}

In next test instance we will validate the contribution of the branch and bound method to the search capability of MOEAs.

{\bf Test instance 5.5} Consider the three-objective optimization problem from \cite{ref34}:
$$F(x)=
\begin{pmatrix}
0.5(x_1^2+x_2^2)+\sin(x_1^2+x_2^2)\\
\frac{(3x_1-2x_2+4)^2}{8}+\frac{(x_1-x_2+1)^2}{27}+15\\
\frac{1}{x_1^2+x_2^2+1}-1.1\exp(-x_1^2-x_2^2)
\end{pmatrix}\quad with\quad x_i\in[-3,3],\;i=1,2.$$

The comparison results are showed as Figure 6. Compared to two pure MOEAs, two hybrid algorithms are able to approximate the Pareto set and Pareto front perfectly. It is worth noting that the hybrid algorithms still outperform the two pure MOEAs when only the mini MOEAs are used, indicating that the branch and bound method enhances MOEAs' search capability, which is achieved by limiting the search region.

\begin{figure}[H]
  \subfigure[]{
    \includegraphics[width=0.23\textwidth]{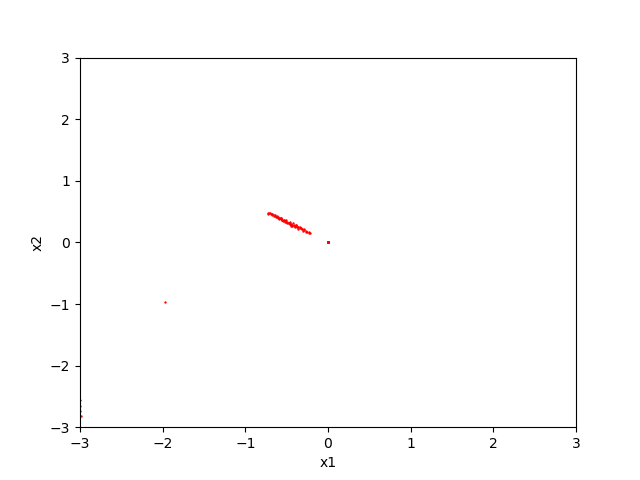}}
  \subfigure[]{
    \includegraphics[width=0.23\textwidth]{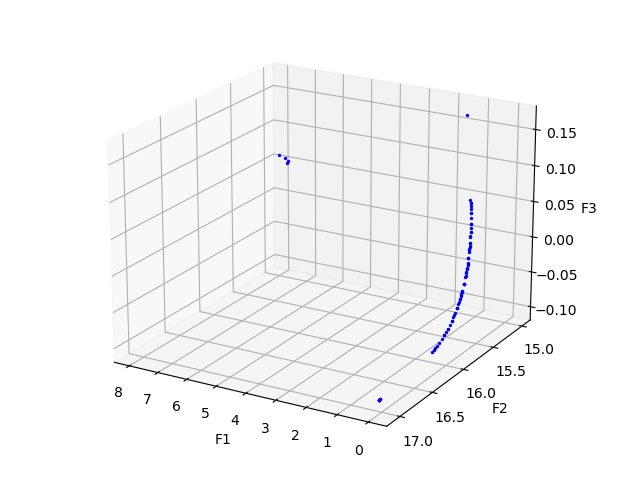}}
    \subfigure[]{
    \includegraphics[width=0.23\textwidth]{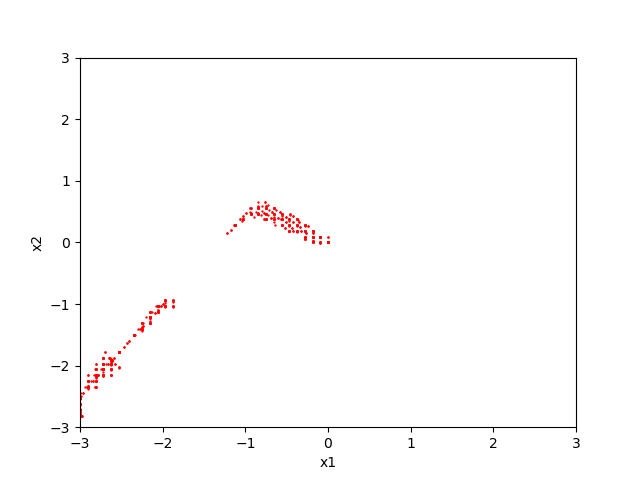}}
  \subfigure[]{
    \includegraphics[width=0.23\textwidth]{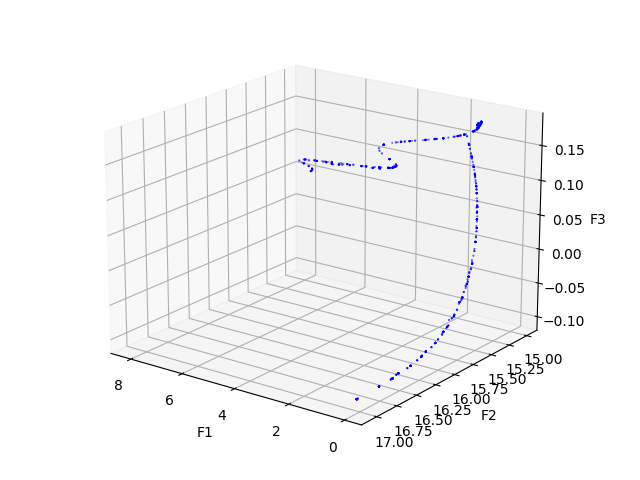}}

  \subfigure[]{
    \includegraphics[width=0.23\textwidth]{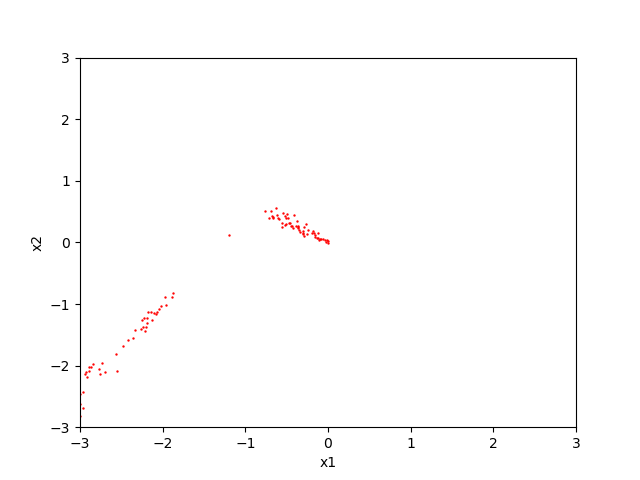}}
  \subfigure[]{
    \includegraphics[width=0.23\textwidth]{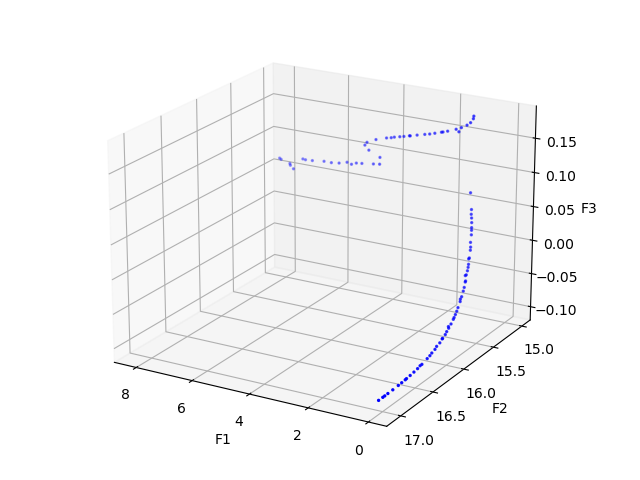}}
  \subfigure[]{
    \includegraphics[width=0.23\textwidth]{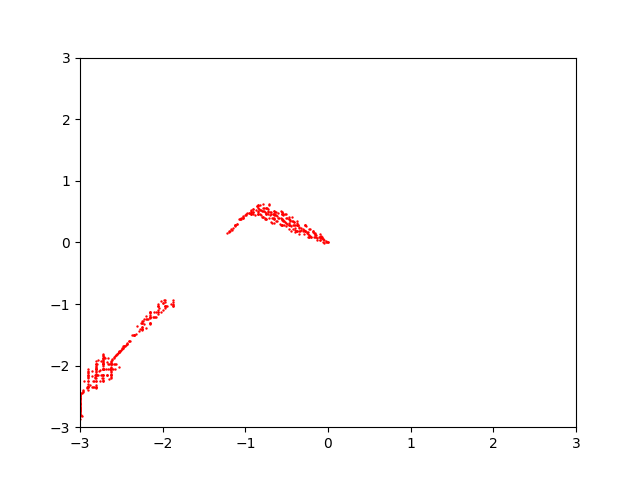}}
  \subfigure[]{
    \includegraphics[width=0.23\textwidth]{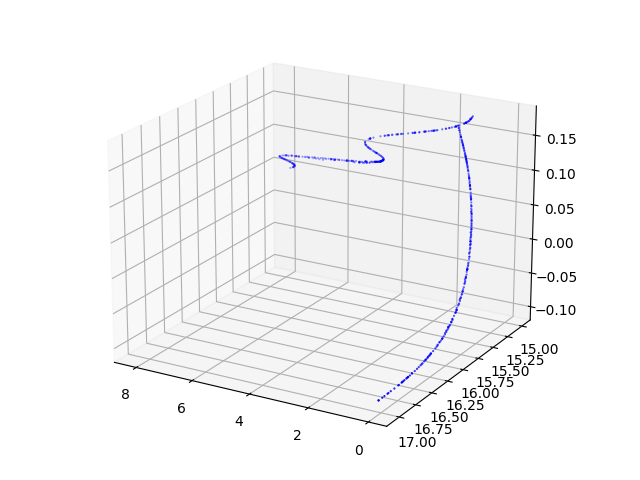}}
    \caption{PBB-MOEA versus MOEA/D-DE and NSGA-II on test instance 5.5.: (a) The solutions generated by MOEA/D-DE; (b) Upper bounds generated by MOEA/D-DE; (c) The solutions generated by PBB-MOEA/D; (d) Upper bounds generated by PBB-MOEA/D; (e) The solutions generated by NSGA-II; (f) Upper bounds generated by NSGA-II; (g) The solutions generated by PBB-NSGAII; (h) Upper bounds generated by PBB-NSGAII.}
  \end{figure}

Finally, we consider a multiobjective optimization problem with nonconvex constraints.

{\bf Test instance 5.6} Consider the constrained optimization problem from \cite{ref10}:
$$F(x)=
\begin{pmatrix}
x_1\\
x_2
\end{pmatrix}$$
subject to the constraints
\begin{align*}
&g_1=x_1^2+x_2^2-1-0.1\cos(16\arctan(x_1/x_2))\geq0,\\
&g_2=0.5-(x_1-0.5)^2-(x_2-0.5)^2\geq0,\\
&0\leq x_1\leq\pi,\\
&0\leq x_2\leq\pi.
\end{align*}

We compare PBB-MOEA/D with MOEA/D-DE and the basic BB algorithm on this instance. The size of population of MOEA/D-DE is 200 and the maximal number of generations is 300, and the penalty coefficient $\rho$ is set to 1. The \emph{natural interval extension} mentioned in \cite{ref27} is employed in the feasibility test.

Since $F$ is the identity, the Pareto optimal set and the Pareto front coincide. Figure 7 provides the comparison of three algorithms on this instance. Unlike in the box constrained cases, the basic BB algorithm does not find enough feasible upper bounds to approximate the complete Pareto front. By contrast, with the mini MOEA, PBB-MOEA/D is able to find more feasible points than the basic BB algorithm does, and it performs similarly to MOEA/D-DE. This also verifies the enhancement of the branch and bound method to the MOEAs' search capability.

\begin{figure}[H]
  \centering
    \subfigure[]{
    \includegraphics[width=0.31\textwidth]{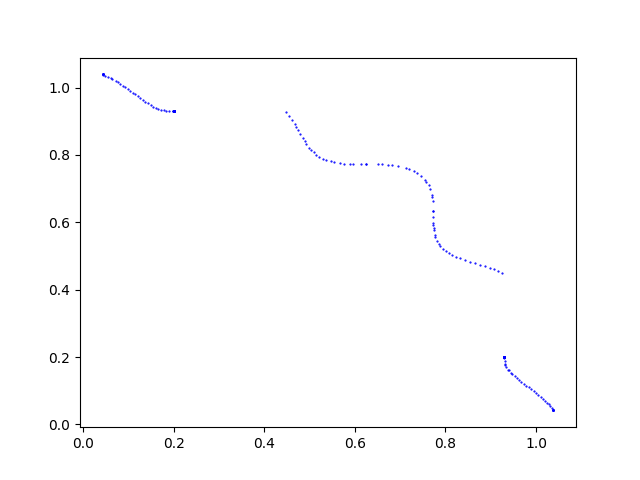}}
  \subfigure[]{
    \includegraphics[width=0.31\textwidth]{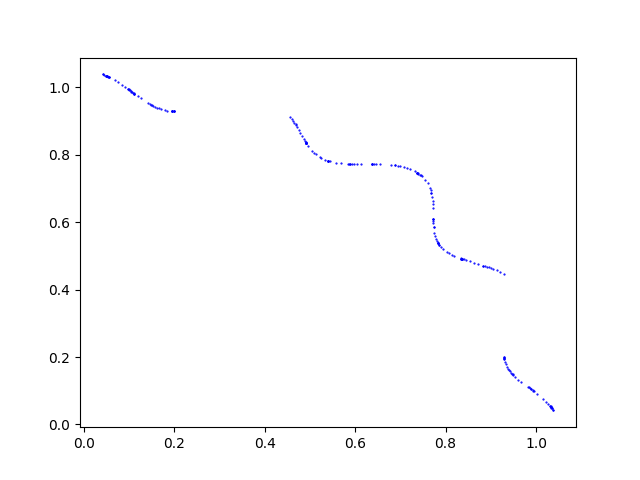}}
  \subfigure[]{
    \includegraphics[width=0.31\textwidth]{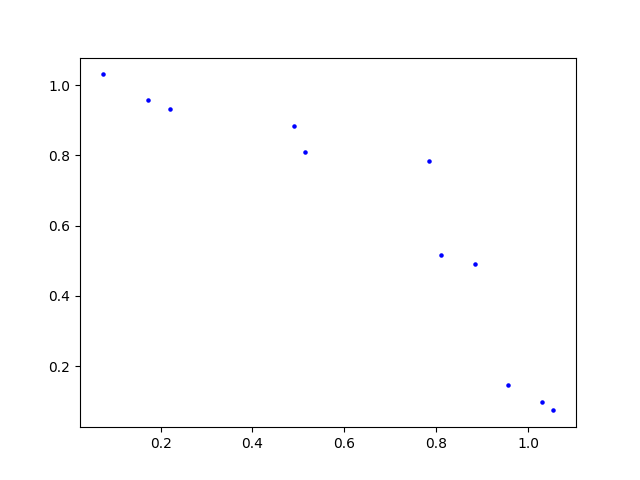}}
  \caption{The comparison of three algorithms on test instance 5.6.: (a) The solutions/feasible upper bounds generated by MOEA/D-DE; (b) The solutions/feasible upper bounds generated by PBB-MOEA/D; (c) The solutions/feasible upper bounds generated by the basic BB algorithm.}
\end{figure}

\section{Conclusions}

An algorithmic framework (PBB-MOEA) which hybrids the branch and bound method with MOEAs is presented for solving nonconvex MOPs. PBB-MOEA brings together the advantages from the branch and bound method and MOEAs. On the one hand, PBB-MOEA has global convergence results and is easy to parallelize which benefits from the branch and bound method. On the other hand, PBB-MOEA is able to obtain tight bounds with the help of the search power of an MOEA, and further the tight bounds can reduce the number of subregions. Numerical experiments show that the resulting hybrid algorithms are more effective in solving multiobjective optimization problems than the branch and bound algorithm or MEOAs.
\begin{acknowledgements}
This work is supported by the Major Program of National Natural Science Foundation of China (Nos. 11991020, 11991024), the General Program of National Natural Science Foundation of China (No. 11971084) and the Natural Science Foundation of Chongqing (No. cstc2019jcyj-zdxmX0016)
\end{acknowledgements}

\end{document}